\documentclass[10pt,a4paper]{article}
\usepackage{don-wc}

\newcommand{\D}{D^\triangleright}
\renewcommand{\d}[2][]{
	\ifthenelse{ \equal {#1} {} }
	{d}
	{d_{#1}}
	\ifthenelse{ \equal {#2} {}}
	{}
	{\left( #2 \right)}	
}
\renewcommand{\dh}[1]{
	\d[\mathscr H]{#1}
}
\newcommand{\cauli}{\overline{W_0}}

\usepackage[backend=bibtex,sorting=nyt,style=numeric,maxbibnames=100]{biblatex} 

\bibliography{don-wc}

\newtheorem{theorem}{Theorem}[section]
\newtheorem{lemma}[theorem]{Lemma}
\newtheorem{corollary}[theorem]{Corollary}

\theoremstyle{definition}
\newtheorem{definition}[theorem]{Definition}

\theoremstyle{remark}
\newtheorem{remark}[theorem]{Remark}

\numberwithin{equation}{section}

\begin{document}

\title{Wandering cauliflowers}


%    author one information
\author{
	{William A. Don}\\
	{The Open University, Milton Keynes, UK}
}
\maketitle

%    \subjclass is required.
%\subjclass[2020]{Primary }

\date{}

%\dedicatory{}

%    Abstract is required.
\begin{abstract}
In this paper we examine an orbit of simply connected wandering domains for the function ${f(z) = z\cos z+2\pi}$.  They are noteworthy in that they are non-congruent but arise from a simple closed form function.  Moreover, the shape of the wandering domains, suitably scaled, converges in the Hausdorff metric to the filled-in parabolic basin of the quadratic ${z^2+c}$ with $c=\tfrac{1}{4}$, commonly named the ``cauliflower''.  

We complete our analysis by classifying the wandering domains within the ninefold framework in \cite{benini+2021}, finding they are contracting and the diameters of the wandering domains tend to zero. 

To conclude we propose an expansion of the analysis to a wider family of functions and discuss some potential results.
\end{abstract}

\maketitle

% SECTION
%==================================================================================================
%==================================================================================================
\section{Introduction}\label{S:INTRO}
The iteration of analytic functions, $f: \Complex \to \Complex$, has been an area of active mathematical research for over a century, with particular investigation of their dynamic behaviour.  The notation $f^n(z) = f(f^{n-1}(z))$, for $n \geq 1$, and $f^0(z) = z$ is well established.  Such functions divide the plane into two completely invariant sets: the Fatou set, $F(f)$, the maximal open set on which the iterates $f^n$ form a normal family, and its complement the Julia set, $J(f)$.  The Fatou set then consists of zero or more connected components.  It is common practice to write $U_n$ for the Fatou component that contains $f^n(U)$.  Then if $U_m = U_{n}$ for minimal $m \geq 0, n \geq 1$ and $n > m$ we say $U$ is periodic ($m=0$) or pre-periodic ($m>0$).  These Fatou components have been thoroughly classified, for instance in \cite{bergweiler93}.  If $U_n \neq U_m$ for all positive pairs $m,n$, then $U$ is a wandering domain.

More recently, much attention has been given to wandering domains.  Sullivan showed in 1982 that no rational functions can have wandering domains \cite{sullivan85}, subsequently broadened to classes of transcendental functions (see, for example, \cite{eremenko+92}).  But wandering domains do exist. The first example was constructed by Baker in 1976 \cite{baker76}.  It was multiply connected and he also showed that every multiply connected Fatou component is a wandering domain, see \cite{baker75} and \cite{baker84}.  

Simply connected wandering domains also exist for transcendental functions; for example the function $f(z) = z +\sin z +2\pi$ \cite[example 2]{fagella2009+} has an orbit of congruent and bounded wandering domains symmetrically placed along the real axis where the $n$th iterate is the $2n\pi$-translate of the wandering domain.  Examples, like this, with closed form expressions, have typically been obtained using a technique in which a function $f:\Complex \setminus \{0\} \to \Complex \setminus\{0\}$ is transformed to a new function $g$ satisfying $\exp(k g(z)) = f(\exp(k z))$, for suitable $k, \abs{k} = 1$. When $g$ is appropriately configured, a simply connected Fatou component of $f$ lifts to a simply connected wandering domain of $g$ (see \cite{bergweiler95} and examples in \cite{marti-pete2021} and \cite{benini+2021}). These functions are often in closed form and the orbit of wandering domains usually consists of congruent components (see \cite[Example 3.5]{marti-pete2021} for an exception).  Another technique, using approximation theory, allows wandering domains with more diverse characteristics to be constructed.  Functions obtained in this way, however, do not usually have a closed form.

In this paper we present a transcendental entire function with simple closed form,  $f(z) = z \cos z + 2\pi$, which has an orbit of non-congruent simply connected wandering domains.  We develop techniques to prove a number of interesting properties of these domains, presented in our main result, Theorem \ref{T:INTRO Main theorem}, below.  Before stating it, we introduce notation used throughout this paper. 

We adopt the standard notation $D(a,b) = \{z \in \Complex : \abs{z-a} < b\}$, where $a \in \Complex$ and $b > 0$ and $\dh{A,B}$ to denote the Hausdorff distance between subsets $A, B$ of a metric space, discussed below in Section \ref{S:CAULI}.
\begin{definition}\label{D:INTRO definition of D_n and D-triangle_n}
	For $n \geq 1$, let $D_n$ be the disc 
		$
		    	D_n = \disc{ \frac{1}{6n\pi},\frac{1}{6n\pi} }
		$
	and 
		\begin{align*}
    		\D_n=\disc{ 2n\pi + \frac{1}{6n\pi},\frac{1}{6n\pi} }
		\end{align*}
	(the triangle indicates translation of $D_n$ to the right along the real axis by $2n\pi$).
\end{definition}
\begin{manualtheorem}{A}
	The function $f: \Complex \to \Complex$ given by $f(z) = z\cos z + 2\pi$ has an orbit of simply connected wandering domains $U_n$ for $n \geq 0$ such that $f(U_n) \subset U_{n+1}$.  There exists $N_0$ such that for $n \geq N_0$ the following properties hold:
\begin{enumerate}[(1)]\label{T:INTRO Main theorem}
\item $\D_n \subset U_n$ and $f(\D_n) \subset \D_{n+1}$;
\item $\mathrm{diam}(U_n) < \frac{2}{n\pi}$; 
\item $U_n$ is symmetrical about the real axis and $2n\pi \in \partial U_n$;
\item\label{T:INTRO:I convergence to cauliflower} $\dh{\overline {V_n}, \cauli} \to 0 $ as $n \to \infty$ where $\cauli$ is the compact set known as the cauliflower and $V_n$ is the component $U_n$ translated $2n\pi$ towards the origin and scaled by the factor $n$;
\item\label{T:INTRO:I nine-fold classification} under the ninefold classification of \cite{benini+2021}, the wandering domains $U_n$ are contracting and iterates converge to the boundary.
\end{enumerate}
\end{manualtheorem}
The wandering domains are illustrated in Figure \ref{F:JULIASET wandering domain}.  The notation in part \ref{T:INTRO:I convergence to cauliflower} is explained with more detail in Sections \ref{S:CAULI} and \ref{S:CONV} below.
\begin{figure}[H]
\centering
	\includegraphics[width=0.95\textwidth]{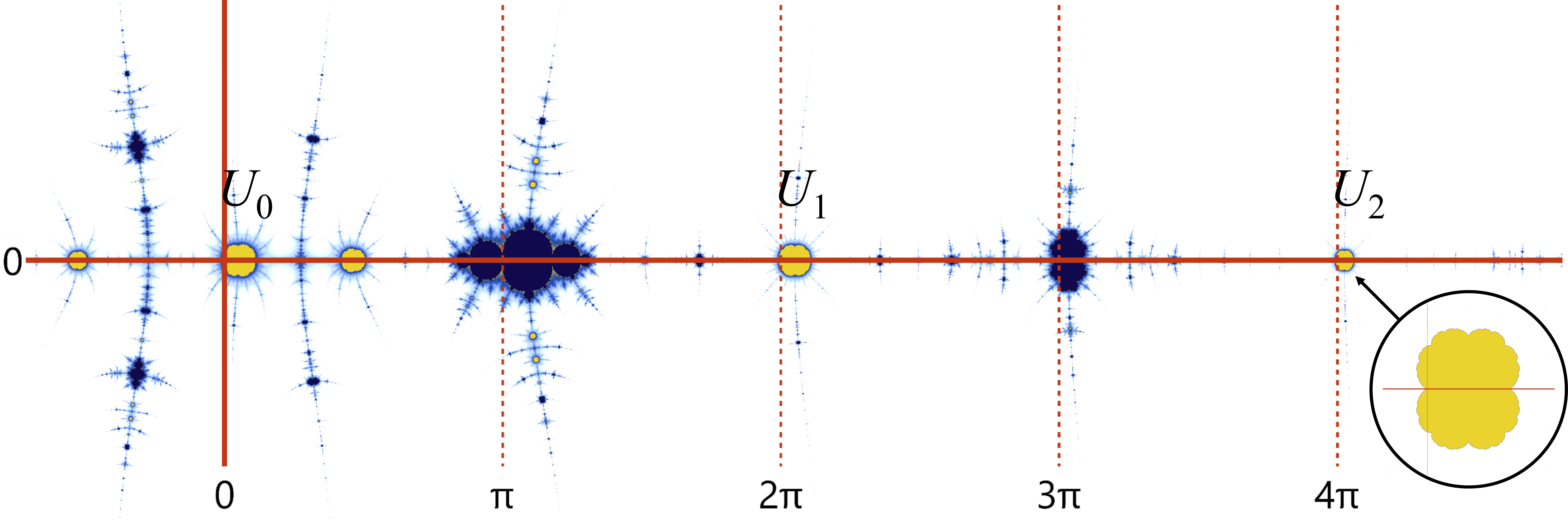}
\caption{The wandering cauliflower and some pre-images are shown in yellow. The dark components are parabolic basins of $f^2$ (in which orbits under $f^2$ converge to odd multiples of $\pi$) and associated pre-images. The wandering component $U_2$ is inset.}\label{F:JULIASET wandering domain}
\end{figure}
We note that the property $\overline{f(z)} = f(\overline z)$ implies the Fatou set is symmetrical about $\Real$.  The first step in the proof of Theorem \ref{T:INTRO Main theorem} is to show that the points $2n\pi$ for $n \in \Integer $ belong to the Julia set of $f$ (Section \ref{S:JULIASET}).  We then show that for large enough positive integers $n$, $f(\D_n) \subset \D_{n+1}$ and each $\D_n$ sits inside $U_n$, a component of $F(f)$ (Section \ref{S:DISC}).  We go on to prove the $U_n$ are distinct, therefore forming an orbit of wandering domains, and prove they are simply connected and bounded (Section \ref{S:WANDER}).  We add pre-images of our $U_{n}$ so that we have a sequence of wandering domains $U_0, U_1, \ldots, U_n, \ldots $. At this point we have established the first three parts of Theorem \ref{T:INTRO Main theorem}.

We now introduce the cauliflower and the Hausdorff metric (Section \ref{S:CAULI}) and prove that, when scaled and translated, the sequence $(U_n)$ converges as $n \to \infty$ to the cauliflower (Section \ref{S:CONV}), proving Theorem \ref{T:INTRO Main theorem}~part~\ref{T:INTRO:I convergence to cauliflower}.  We finish with the proof of Theorem~\ref{T:INTRO Main theorem}~part~\ref{T:INTRO:I nine-fold classification} in Section \ref{S:DYNA}.

Finally, in Section \ref{S:CONCLUDE}, we conjecture that this work may extend to a wider family of transcendental entire functions, $f_\lambda(z) = z\cos z + \lambda \sin z + 2\pi$ for $\lambda \in \Complex$, for which we make preliminary observations and show computer images but postpone further discussion to future work.
%==================================================================================================
% SECTION																						   
%==================================================================================================
%==================================================================================================
\section{Points in the Julia set of $f$}\label{S:JULIASET}
In this section we show that the points $2n\pi$ are members of $J(f)$.  We begin by using repelling fixed points, which always belong to $J(f)$ \cite[Section 3.1]{bergweiler93}.  The following lemma provides a supply of these.
\begin{lemma}\label{L:JULIASET: supply of repelling points}
	The function $f(z)=z\cos z + 2\pi$ has an indifferent fixed point at $z = \pi$ with multiplier $-1$.  All other real fixed points are repelling, positioned at $z=4\pi/3$, and thereafter, for $n = 1,2, \ldots$, at $z=2n\pi + \eta_{2n}$ and $z=2(n+1)\pi - \eta_{2n+1}$, where $0 < \eta_{2n}, \eta_{2n+1} < \tfrac{\pi}{2}$ and $\eta_k \to 0 $ as $k\to\infty$.
\end{lemma}
\begin{proof}
	The function $f$ has a real fixed point $x \neq 0$ when $\cos x = 1 - \frac{2\pi}{x}$.  Thus for $x \leq 2\pi $, there are two fixed points, one at $x =\pi$, multiplier $-1$, and a repelling fixed point at $x = 4\pi/3$.  For $x > 2\pi$, there are two fixed points in each open interval $\big(2n\pi, 2(n+1)\pi\big)$ for $n \geq 1$, which may be written $2n\pi + \eta_{2n}$ and $2(n+1)\pi - \eta_{2n+1}$ where
		\begin{align*}
			0 < \eta_{2n},\eta_{2n+1} < \frac{\pi}{2}.
		\end{align*}
	Since $\frac{2\pi}{x}$ tends to zero as $x \to \infty$, we have $\lim_{k\to\infty} \eta_k = 0$. At a fixed point $x > 2\pi$, $\cos x = 1 - \frac{2\pi}{x}$ and the multiplier is $f'(x) = \cos x - x \sin x$, so
		\begin{align*}
			\abs{f'(x)} 
			 	&	\geq \abs{x \sin x} - \abs{\cos x}
				\\&	\geq x\sqrt{1 - \left(1 - \frac{2\pi}{x}\right)^2} - 1 
				\\&	= 2\sqrt{ x \pi - \pi^2 } - 1 
				\\&	> 2\pi-1.
		\end{align*}
	Thus all real fixed points greater than $2\pi$ are repelling.
\end{proof}
We use this supply of points in $J(f)$ to show that $0 \in J(f)$.  We first show that there exist  arbitrarily small negative real points whose iterates eventually move away from a multiple of $2\pi$.
\begin{lemma}\label{L:JULIASET: iteration of negative real expands to interval}
	For all $\delta >0$ there exists a real $x_0 \in (-\delta,0)$ with $f^n(x_0) \leq 2n\pi - \frac{\pi}{2}$ for some $n \geq 2$.
\end{lemma}
\begin{proof}
	Without loss of generality we may suppose $0 < \delta < \frac{\pi}{2}$.  
	
	Aiming for a contradiction, suppose that for all $x \in (-\delta,0)$ and all $n \geq 0$, we have 
	\begin{align}\label{E:JULIASET: f^n(z) > 2 n pi - pi/2 in proof negative z iterates to interval}
		f^n(x) > 2n\pi - \pi/2.
	\end{align}
	Choose a point $x_0 \in (-\delta,0)$ and create the sequence $x_n = f^n(x_0)-2n\pi$.  Clearly the $x_n$ are real and by \eqref{E:JULIASET: f^n(z) > 2 n pi - pi/2 in proof negative z iterates to interval}, $x_n > -\frac{\pi}{2}$ for all $n \geq 0$. Rearrange the definition of $x_n$ to obtain for $n \geq 0$
			\begin{align}\label{E:JULIASET: x_n recurrence formula for initial negative x}
			    x_{n+1} = x_n + (2n\pi + x_n)(\cos x_n - 1)
			\end{align}
	whence $x_1 < 0$ and thereafter for $n \geq 1$, $x_{n+1} < x_n$.  Thus the sequence becomes decreasing and is bounded below, and so has a limit $x \in [-\frac{\pi}{2}, x_1)$.  But writing equation \eqref{E:JULIASET: x_n recurrence formula for initial negative x} in the form
		\begin{align*}
		    \frac{x_{n+1} + 2n\pi}{x_n+2n\pi} = \cos x_n
		\end{align*}
	and letting $n \to \infty$ we find $\cos x = 1$ which gives a contradiction because $ -\frac{\pi}{2} \leq x < 0$.  Thus the assumption is false and there exists some $x_0 \in (-\delta,0)$ and $n \geq 0$ for which \eqref{E:JULIASET: f^n(z) > 2 n pi - pi/2 in proof negative z iterates to interval} is false.  The inequality is always true when $n=0$ or $1$ so the required $n$ must be $2$ or more.
\end{proof}
We can now deduce the following.
\begin{lemma}\label{L:JULIASET: 2npi in the J(f)}
	For all $n \in \Integer$, $2n\pi \in J(f)$.
\end{lemma}
\begin{proof}
	We show first $0 \in J(f)$.  If not then $0 \in F(f)$ and as $F(f)$ is open and invariant, there exists a neighbourhood $V$ of the origin such that $f^n(V)$ is in $F(f)$ for all $n$.  But by Lemma \ref{L:JULIASET: iteration of negative real expands to interval}, every neighbourhood of $0 $ contains a small negative real $x_0$ with $n \geq 2$ for which $f^n(x_0) \leq 2n\pi -\pi/2$.  Then the image under $f^n$ of the interval $[x_0,0]$ contains $f^n(0) = 2n\pi$ and $f^n(x_0) \leq 2n\pi - \pi/2$, and by continuity all points between.  But Lemma \ref{L:JULIASET: supply of repelling points} implies that this interval also contains a repelling fixed point $p$ of $f$, where $p = 2n\pi - \eta_{2n-1}$;  neither end point equals $p$ and so there exists $\xi \in (x_0,0)$ such that $f^n(\xi) = p$.  Every repelling fixed point is in $J(f)$, and by the complete invariance of $J(f)$, it follows $\xi \in J(f)$.  This contradicts the supposition that $V \subset F(f)$ and so $0 \in J(f)$.  
	
	All points $2n\pi$ are either images or pre-images of $0$ under $f^n$,  and so by complete invariance, $2n\pi \in J(f)$ for $n \in \Integer$.  
\end{proof}
%==================================================================================================
% SECTION																						   
%==================================================================================================
%==================================================================================================
\section{Behaviour of $f^n$ on a disc}\label{S:DISC}
In Definition \ref{D:INTRO definition of D_n and D-triangle_n} we introduced contracting discs $\D_n$.  For large enough $n$ we now show that $f$ maps $\D_n$ into its successor, a fact which we later use to establish a wandering domain for $f$.  However, rather than work with $f^n$ directly it is convenient to use two changes of variable, prompting the following definitions.
\begin{definition}\label{D:DISC define C_n, T, h_n and psi_m,n, w_n and H_a}
	Define the following:
	\begin{enumerate}[(a)~]
	\item
	the map $T: \Complex \to \Complex$ to be the translation $z \mapsto z+2\pi$;
	\item for integers $n \geq 0$,
		\begin{align*}
			h_n(z) &= T^{-(n+1)}\circ f \circ T^n(z) = (z+2n\pi)\cos z - 2n\pi;
		\end{align*}
	\item
	for $m,n \geq 0$, 
		\begin{align*}
			\psi_{m,n}(z) =  
			\begin{dcases}
				z, &\text{for } m \geq 0, n = 0 \\
				h_{m+n-1} \circ h_{m+n-2}\circ \cdots \circ h_m(z), & \text{for } m\geq 0, n > 0;
			\end{dcases} 
		\end{align*}
	\item for $m\geq 0$, the family $\mathscr P_m = \{\psi_{m,n}: n \geq 0\}$;
	\item
	for real $a \geq 0$, $H_a = \{z \in \Complex: \Re z > a \}$; 
	\item for $n > 0$ and $t \in H_0$, $\displaystyle w_n(t) = \frac{1}{h_n(1/t)}$; and 
	\item
	for $n > 0$, the circle $C_n = \left \{ z \in \Complex:  \abs{z} = \frac{2}{n\pi} \right \}$.
	\end{enumerate}
\end{definition}
\begin{remark}
	We state here a number of elementary consequences.  For $m,n \geq 1$, 
		\begin{align*}
	    	\psi_{m,n}(z) = T^{-(m+n)}\circ f^n \circ T^m(z).
		\end{align*}
	The cosine power series gives for $n \geq 1$ and all $z$
		\begin{align}\label{E:DISC: power series for h_n}
			h_n(z)
				&		= \sum_{k=0}^{\infty} (-1)^k \left(\frac{z^{2k+1}}{(2k)!} - \frac{2n\pi z^{2k+2}}{(2k+2)!}\right) 
				\\&		= z - \frac{2n\pi}{2!} z^2 - \frac{1}{2!} z^3 
						+ \frac{2n\pi}{4!} z^4 + \frac{1}{4!}z^5 - \frac{2n\pi}{6!} z^6 - \frac{1}{6!} z^7  +  \cdots 
						\nonumber .
		\end{align}
	For $a > 0 $, $H_a = 1/\disc{\tfrac{1}{2a},\tfrac{1}{2a}}$.
\end{remark}
We now follow Fatou's approach to analysing the dynamical behaviour of iterates near a parabolic fixed point and examine the behaviour of the function $w_n$ on a suitable half-plane.
\begin{lemma}\label{L:DISC w_n moves to right on half-plane}
	For $n \geq  1$ and $t \in H_{3n\pi}$ the following inequalities hold
		\begin{align}\label{E:DISC real part t inequality}
	    	\Re t + \frac{9}{20}n\pi < \Re w_n(t) < \Re t + \frac{27}{17}n\pi 
		\end{align}
and if $a \geq 3n\pi$, $w_n(H_a) \subset H_{a+\frac{9}{20}n\pi}$.
\end{lemma}
\begin{proof}
	From Definition \ref{D:DISC define C_n, T, h_n and psi_m,n, w_n and H_a} and the power series \eqref{E:DISC: power series for h_n} we obtain for non-zero $t$
		\begin{align*}
		    w_n(t) 
		    	&= \frac{1}{h_n(1/t)} 
		    	\\&= 
		    		\frac{t}{1 - \frac{n\pi}{t} - \frac{1}{2!t^2} + \frac{2n\pi}{4!t^3} + \frac{1}{4!t^4} - \frac{2n\pi}{6!t^5} - \frac{1}{6!t^6} + \cdots}.
		\end{align*}
	Let $\displaystyle 
		    s = n\pi  + \frac{1}{2! t} - \frac{2n\pi}{4!t^2} - \frac{1}{4!t^3} + \frac{2n\pi}{6!t^4} + \frac{1}{6!t^5} - \cdots $,
	so that $\displaystyle w_n(t) = \frac{t}{1-\frac{s}{t}} $.
	
	We show $\abs{\frac{s}{t}} < 1$.  For $n\geq1$ and $a\geq 3n\pi$, if $t\in H_a$ then $\abs t \geq \Re t > 3n\pi \geq 9$, and
		\begin{align}\label{E:DISC bound on abs s}
		    \abs{s} &\leq n\pi\left(
		     1 + \frac{1}{2! n\pi \abs{t}} 
		    		    	+ \frac{2}{4! \abs{t}^2}
		    		    	+ \frac{1}{4! n\pi \abs{t}^3}
		    		    	+ \frac{2}{6! \abs{t}^4} 
		    		    	+ \frac{1}{6! n\pi \abs{t}^5}
		    		    	+\cdots
		    \right) \nonumber \\
		    &< n\pi\left(1+\frac{1}{\pi \abs{t}} + \frac{1}{\pi^2 \abs{t}^2 }+ \frac{1}{\pi^3 \abs{t}^3} + \cdots\right) \nonumber \\
		    &< \frac{27}{26}n\pi.
		\end{align}
	It follows that $\displaystyle \abs{\frac{s}{t}} < \frac{27}{26}n\pi\cdot \frac{1}{3n\pi} = \frac{9}{26}< 1 $.

	Accordingly we may expand $\displaystyle \frac{1}{1-\frac{s}{t}}$ and write for $t \in H_a$ where $a \geq 3n\pi$
		\begin{align*}
		    w_n(t) = t\left(1+\frac{s}{t} + \frac{s^2}{t^2}+ \cdots\right).
		\end{align*}
	Taking real parts,
		\begin{align}\label{E:DISC real part of w_n}
		    \Re w_n(t) = \Re t + \Re s + \Re\left(
		    	\frac{s^2}{t}
		    	+\frac{s^3}{t^2} 
		    	+ \cdots
		    \right).
		\end{align}
	Evaluating the second and third terms individually, using $\abs{\Re z} \leq \abs{z}$ for all $z$,
		\begin{align*}
		    \Re s 
		    &= n\pi + \frac{\Re t}{2\abs{t}^2} -
		    	\Re 
		    		\left(
		    			\frac{n\pi}{12 t^2} 
		    			\left(
		    				1 
		    				+ \frac{1}{2n\pi t} 
		    				-\frac{1}{6\cdot 5 \cdot  t^2} 
		    				-\frac{1}{6\cdot5\cdot2 n\pi t^3} 
		    				+\cdots
		    			\right)
		    		\right) \\
		    &\geq n\pi + \frac{\Re t}{2\abs{t}^2} 
		    	- \frac{n\pi}{12\abs{t}^2}
		    	\left(
			    	1 
			    	+ \frac{1}{5\abs{t}} 
			    	+ \frac{1}{5^2\abs{t}^2} 
			    	+ \frac{1}{5^3\abs{t}^3} 
			    	+ \cdots
		    	\right) \\
		    &\geq  n\pi 
		    	+ \frac{n\pi}{\abs{t}^2}
		    		\left(
		    			\frac{3}{2} - \frac{1}{12} \cdot \frac{45}{44}
		    		\right) \\
		    &> n\pi
		\end{align*}
	and
		\begin{align*}
		    \abs{\Re\left( \frac{s^2}{t} + \frac{s^3}{t^2} + \cdots\right) }
		    &\leq
		    \abs{\frac{s^2}{t}}
		    	\left(
		    		1 + \abs{\frac{s}{t}} + \abs{\frac{s}{t}}^2 + \cdots
		    	\right) \\
		    &\leq
		    \left(\left( \frac{27}{26}n\pi\right)^2 \cdot \frac{1}{3n\pi}\right)\cdot\left(\frac{1}{1-\frac{9}{26}}\right) \\
		    &=
		    \frac{9}{17}\cdot\frac{27}{26}n\pi \\
		    &\leq
		    \frac{11}{20}n\pi .
		\end{align*}
	From these inequalities and equation \eqref{E:DISC real part of w_n} we obtain both the first part of \eqref{E:DISC real part t inequality} and the set inclusion for $a \geq 3n\pi$
		\begin{align*}
		    \Re w_n(t) &> \Re t + \frac{9}{20}n\pi, \quad w_n(H_a) \subset H_{a+\frac{9}{20}n\pi}.
		\end{align*}
	We also know $\Re t $ is positive so that from equations \eqref{E:DISC bound on abs s} and \eqref{E:DISC real part of w_n} we obtain the upper bound 
		\begin{align*}
		    \Re w_n(t) \leq \Re t + \frac{27}{26}n\pi + \frac{27}{26}\cdot\frac{9}{17}n\pi < \Re t + \frac{27}{17}n\pi
		\end{align*}
	and the second part of \eqref{E:DISC real part t inequality} is proved.  
\end{proof}
We use Lemma \ref{L:DISC w_n moves to right on half-plane} to show that $h_n$ maps $D_n$ into its successor $D_{n+1}$.  There is flexibility in choosing the rate at which the discs shrink -- we could equally have constructed discs shrinking in proportion to $1/n^2$ -- but the slower rate of contraction given by $1/n$ will be useful later.  Our result follows from a simple corollary valid for larger $n$.
\begin{definition}\label{D:DISC definition N_0}
	Define $N_0 = 7$.
\end{definition}
\begin{corollary}\label{C:DISC inclusions for H_3n pi}
	For $n \geq N_0$, $w_n(H_{3n\pi}) \subset H_{3(n+1)\pi}$.
\end{corollary}
\begin{proof}
	Take $a = 3n\pi$ in Lemma \ref{L:DISC w_n moves to right on half-plane}, so $w_n(H_{3n\pi})\subset H_{3n\pi+\frac{9}{20}n\pi}$.  When $n \geq N_0$, $\frac{9}{20}n\pi > 3\pi$.  
\end{proof}
\begin{lemma}\label{L:DISC: inclusion for h_n}
	For all $n \geq N_0$, $ h_n(D_n) \subset D_{n+1}$ and $f(\D_n)\subset \D_{n+1}$.
\end{lemma}
\begin{proof} 
	From Corollary \ref{C:DISC inclusions for H_3n pi}, $w_n(H_{3n\pi}) \subset H_{3(n+1)\pi}$. We can now conjugate $w_n$ using the M\"obius mapping $t \mapsto z = \frac{1}{t}$; each of the half-planes $H_{3n\pi}, H_{3(n+1)\pi}$ corresponds respectively to the open discs $D_n, D_{n+1}$.  
	
	But $\displaystyle \frac{1}{w_n(1/t)} = h_n(t)$, so $ h_n(D_n) \subset D_{n+1} $ and from the definition of $h_n$, $f(\D_n)\subset \D_{n+1}$.
\end{proof}
The results above are illustrated in the commutative diagram Figure \ref{F:DISC commutative diagram} for $m \geq N_0$.
\begin{figure}[H]
	\centering
	\begin{equation*}
	    \begin{tikzcd}[column sep=30pt, row sep=30pt, 
	    	% /tikz/column /.style={column sep=4pt} 
	    	]
			H_{3m\pi} \arrow[d, "\frac{1}{z}"] \arrow[r, "w_m"] 
				& H_{3(m+1)\pi} \arrow[d, "\frac{1}{z}"] \arrow[r, "w_{m+1}"] 
				&  \cdots \arrow[r, "w_{m+n-1}"]
				& H_{3(m+n)\pi} \arrow[d, "\frac{1}{z}"] \arrow[r, "w_{m+n}"] 
				&  \cdots
			\\ 	D_{m} \arrow[d,"T^m"] \arrow[r, "h_m"] 
				& D_{m+1} \arrow[d, "T^{m+1}"] \arrow[r, "h_{m+1}"] 
				& \cdots \arrow[r, "h_{m+n-1}"]
				& D_{m+n} \arrow[d, "T^{m+n}"] \arrow[r, "h_{m+n}"]
				&  \cdots
			\\ 	\D_{m} \arrow[r, "f"] 
				& \D_{m+1} \arrow[r, "f"] 
				& \cdots \arrow[r,"f"]
				& \D_{m+n} \arrow[r,"f"]
				& \cdots
		\end{tikzcd}
	\end{equation*}
	\caption{Commutative diagram for $H_{3(m+n)\pi}$, $D_{m+n}$ and $\D_{m+n}$}\label{F:DISC commutative diagram}
\end{figure}
We now deduce the following.
\begin{corollary}\label{C:DISC psi and fn normal on appropriat D_m}
	For $m \geq N_0$ the family $\mathscr P_m = \{\psi_{m,n}: n \geq 0\}$ is normal on $D_m$ and the family $\{f^n: n \geq 0\}$ is normal on $\D_m$.
\end{corollary}
\begin{proof}
	For the first part, Lemma \ref{L:DISC: inclusion for h_n} showed that for $m \geq N_0$ and all $k \geq 0$, $h_{m+k} (D_{m+k}) \subset D_{m+k+1}$.  Moreover, by definition $D_{m+n} \subset D_m$, so that
			\begin{align*}
			    \psi_{m,n}(D_m) = h_{m+n-1}\circ h_{m+n-2}\circ \cdots\circ h_m (D_m) \subset D_{m+n} \subset D_m.
			\end{align*}
	For completeness, $\psi_{m,0}(D_m) = D_m$.  Therefore the images of $D_m$ under all members of the family $\mathscr P_m$ are contained in $D_m$ and exclude all points outside $D_m$.  Montel's Theorem \cite[\textsection 2, p. 3]{steinmetz} then implies $\mathscr P_m$ is normal on $D_m$.

	For the second part, we have $T^{-(m+n)} \circ f^n \circ T^m = \psi_{m,n}$ and therefore Lemma \ref{L:DISC: inclusion for h_n} implies
		\begin{align*}
		    f^n(\D_m) \subset \D_{m+n}.
		\end{align*}
	By construction, all even multiples of $2\pi$ are excluded from all $D_n$ and therefore also from all $\D_n$ (and there are many other points we could have chosen here).  Applying Montel's Theorem again, the family $\{f^{n}: n \geq 0\} $ is normal on $\D_m$.
\end{proof}
It follows then that for $n \geq N_0$, each $\D_n$ is in $F(f)$; we define $U_n$ to be the component of $F(f)$ that contains $\D_n$.  Because $\D_n$ maps into $\D_{n+1}$ and components are connected, it follows that $f(U_n) \subset U_{n+1}$.  We can also define $U_n$ for $n < N_0$ such that $f(U_n)\subset U_{n+1}$.  For example, take  $n=N_0-1$ and consider the preimage of $U_{N_0}$.  It is a subset of $F(f)$ \cite[Lemma 2]{bergweiler93} and because $f$ is entire, $f(2n\pi) = 2(n+1)\pi$ and $f'(2n\pi)= 1$, this preimage has one component that contains the real interval $(2n\pi, 2n\pi+\delta)$ for a small $\delta>0$.  This component we label $U_n$. The process can be repeated successively for smaller $n$, leading to the following lemma.
	\begin{lemma}
			The function $f$ has Fatou components $U_n$, for $n \in \Integer$, satisfying $f(U_n)\subset U_{n+1}$ and where each contains a real interval $(2n\pi, 2n\pi+\delta)$ for some $\delta>0$ (dependent on $n$).  
			
			For $n \geq N_0$ (but not necessarily for smaller $n$) $U_n$ contains $\D_n$. 
	\end{lemma}
We will show that the components $U_n$, $n \geq N_0$ are distinct shortly. First we obtain a second inequality for $h_{m+n}$ on the circles $C_{m}$ (Definition \ref{D:DISC define C_n, T, h_n and psi_m,n, w_n and H_a}) which will be applied later to show that the $U_n$ are bounded.  
\begin{lemma}\label{L:DISC show that the h_ns expand a circle}
		For all $m \geq 1$ and $n \geq 0$
			\begin{align*}
			    \abs{h_{m+n}(C_m)} > \frac{2}{m\pi}.
			\end{align*}
\end{lemma}
\begin{proof}
		We use the power series \eqref{E:DISC: power series for h_n} for $h_{m+n}$.  Let $r = \frac{2}{m\pi}$, $z = r e^{i \theta}$, and express $h_{m+n}$ as a sum
			\begin{align*}
			    h_{m+n}(z) = A_{m+n}(z) + R_{m+n}(z)
			\end{align*}
		where $A_{m+n}(z)$ is the first three terms of \eqref{E:DISC: power series for h_n}, which we can break into real and imaginary parts,
			\begin{align*}
			    A_{m+n}(z) &= r \cos \theta - (m+n)\pi r^2 \cos 2\theta - \frac{1}{2}r^3 \cos 3\theta \\
			    &\hspace{1cm}
			    + i\left(r\sin\theta - (m+n)\pi r^2 \sin 2\theta - \frac{1}{2}r^3 \sin 3\theta \right)
			\end{align*}
		and $R_{m+n}(z)$ is the remainder
			\begin{align*}
			    R_{m+n}(z) = \frac{2(m+n)\pi z^4}{4!} + \frac{z^5}{4!} - \frac{2(m+n)\pi z^6}{6!} - \frac{z^7}{6!} + \cdots .
			\end{align*}
		From the first, we use trigonometric identities to obtain 
			\begin{align*}
			    &\abs{A_{m+n}(z)}^2 
			    \\&\hspace{1cm}
				= r^2 \bigg(
					1 
						+ (m+n)^2 \pi^2 r^2 
						+ \tfrac{1}{4}r^4 
						- 2(m+n) \pi r\left(1-\tfrac{1}{2}r^2\right) \cos\theta 
						- r^2 \cos 2\theta \bigg).
			\end{align*}
		For fixed $r=\frac{2}{m\pi}$, the coefficients of both $\cos \theta$ and $\cos 2\theta$ are negative and the minimum of $\abs{A_{m+n}}^2$ is obtained when $\cos\theta = 1$.  Substituting the value of $r$ to rewrite terms in $(m+n)\pi r$, we obtain,
			\begin{align*}
			    \abs{A_{m+n}(z)}^2 
			    	&\geq r^2 \bigg(
			    			1 + \left(2 + \tfrac{2n}{m}\right)^2 - 2 \left(2 + \tfrac{2n}{m}\right) + \tfrac{1}{4}r^4 + \left(2 + \tfrac{2n}{m}\right)r^2 - r^2
			    		\bigg)
			    	\\&= r^2\bigg(
			    		\Big(1 + \tfrac{2n}{m}\Big)^2 + \tfrac{1}{4}r^4 + r^2 + \tfrac{2n}{m}r^2
			    	\bigg).
			\end{align*}
		Turning to $R_{m+n}$, we use $r < \frac{2}{3}$ to sum the geometric series below, giving
			\begin{align*}
			    \abs{R_{m+n}(z)} 
			    &\leq \frac{2(m+n)\pi r^4}{4!} + \frac{r^5}{4!} +\frac{2(m+n)\pi r^6}{6!} + \frac{r^7}{6!} + \cdots \\
			    &\leq \frac{2(m+n)\pi r^4}{4!} \left( 1+ \tfrac{r}{5} + \tfrac{r^2}{5^2} + \tfrac{r^3}{5^3 } + \cdots\right) \\
			    &< \left(1+\tfrac{n}{m}\right)\cdot \frac{r^3}{5}.
			\end{align*}
		To complete the proof we show that the quantity 
			\begin{align*}
			    \Delta(z) = \frac{1}{r^2} \Big( \abs{A_{m+n}(z)}^2 - (\abs{R_{m+n}(z)} + r \big)^2 \Big)
			\end{align*}
		is positive.  Gathering terms
			\begin{align*}
			    \Delta(z) &\geq
			    	\left\{ 1 + \frac{4n}{m} + \frac{4n^2}{m^2} + \tfrac{1}{4}r^4 + r^2 + \frac{2n}{m}r^2 \right\}
			    	\\&\hspace{2cm}
			    	-\bigg\{
			    		\left(\left(1 + \frac{n}{m} \right)\cdot \frac{r^2}{5 } \right)^2 +
			    		2 \left(\left(1 + \frac{n}{m} \right)\cdot \frac{r^2}{5} \right) +
			    		1
			    	\bigg\} \\
			    	&= \frac{n^2}{m^2}\left( 4 - \frac{r^4}{5^2}\right) 
			    	+ \frac{n}{m}\left( 4 + 2r^2 - \frac{2r^4}{5^2} - \frac{2  r^2}{5}\right) 
			    	\\ & \hspace{2cm}
			    	+ r^4 \left(\frac{1}{4} - \frac{1}{5^2}\right)
			    	+ r^2 \left(1 - \frac{2}{5} \right) .
			\end{align*}
		Here, $r < 1$ so all four terms on the right are positive, and $\Delta(z) > 0$, for $\abs{z} = r$, implying $\displaystyle \abs{h_{m+n}(z)} \geq \abs{A_{m+n}(z)} - \abs{R_{m+n}(z)} > r $ as required.
\end{proof}
%==================================================================================================
% SECTION																						   
%==================================================================================================
%==================================================================================================
\section{Wandering domains}\label{S:WANDER}
In this section we show that the Fatou components $U_0, U_1, \ldots, $ are simply connected, mutually distinct and for $n \geq N_0$, bounded.  The next lemma deals with the first two properties.
\begin{lemma}\label{L:WANDER the U_n are simply connected and distinct}
		Every $U_n$, $n \in \Integer$, is a simply connected wandering domain of $f$.  
\end{lemma}
\begin{proof}
		We prove the simple connectedness of $U_n$ and that for $n \neq m$, $U_n \cap U_m = \emptyset$ by contradiction.
	
		Suppose first that for some $k \in \Integer$, $U_k$ is multiply connected. Then from theorems of Baker \cite[Theorem~3.1]{baker84} and Zheng \cite[Theorem 5]{zheng2006} $U_k$ is a wandering domain and for all sufficiently large $n$, $f^n(U_k)$ contains an annulus $A_n = \{ z: r_n < \abs{z} < R_n \}$ where both $r_n, R_n \to \infty$ and $R_n/r_n \to \infty$.  This implies $f^n(U_k)$ must eventually contain a segment of the real axis of length more than $2\pi$ and so include an integer multiple of $2\pi$.  But all such points are in $J(f)$ (Lemma \ref{L:JULIASET: 2npi in the J(f)}) leading to a contradiction.  Thus every $U_k$ must be simply connected.
		
		Next, if for $m, n \in \Integer$, $m < n$, $U_n$ and $U_m$ not distinct then they are a single Fatou component $U = U_n = U_m$.  By construction $U_m$ and $U_n$ contain small open intervals on the real axis with left end-points at $2m\pi$ and $2n\pi$ respectively.  It follows $U$ contains both such intervals.  On the other hand $U$ is symmetric about the real axis, because $f$ is, and therefore, as $U$ is connected it must surround the point $2(m+1) \pi \in J(f)$.  This implies $U$ is multiply connected, contradicting our earlier conclusion.
		
		Therefore whenever $m < n$, $U_m$ and $U_n$ are distinct, and because $f(U_m) \subset U_{m+1}$, the orbit of $U_m$ never repeats and $U_m$ is a wandering domain for all $m \in \Integer$.
\end{proof}
We can now prove that the Fatou components $(U_n)$ are bounded for large $n$. The approach is similar to one used in \cite[Example 2 and Lemma 7]{ripponandstallard2008}.
\begin{lemma}\label{L:WANDER U_n are bounded}
	For $n \geq N_0$ each Fatou component $U_n$ is contained in the translated circle $T^n (C_n)$.
\end{lemma}
\begin{proof}
	Fix $m \geq N_0$.  Recall Definition \ref{D:DISC define C_n, T, h_n and psi_m,n, w_n and H_a} and Corollary \ref{C:DISC psi and fn normal on appropriat D_m}, which showed that $\mathscr P_m = \{\psi_{m,n}: n \geq 0\}$ is normal on $D_m$.  
	
	It follows $\mathscr P_m$ is also normal on $T^{-m}(U_m)$.  To prove this, we show that the image of $T^{-m}(U_m)$ under members of $\mathscr P_m$ excludes multiples of $2\pi$.  For, if not, then for some $m \geq N_0$, there exits $n \geq 0$, an integer $k$ and $z \in T^{-m}(U_m)$ such that $\psi_{m,n}(z) = 2k\pi$.  This means $ T^{-(m+n)}\circ f^n\circ  T^m(z) = 2k\pi$.  Put $\zeta = T^m(z) \in U_m$, so that $f^n(\zeta) = 2(k+m+n)\pi $.  But this means that $\zeta \in U_m \subset F(f)$ maps into the Julia set of $f$, which is impossible.  Thus all multiples $2k\pi$ must be excluded from the images of $T^{-m} (U_m)$ under each member of $\mathscr P_m$ and by Montel's Theorem $\mathscr P_m$ is normal on $T^{-m}(U_m)$.
	
	Now as $n\to\infty$ the discs $D_{m+n}$ shrink to zero, and since $\psi_{m,n}(D_m) \subset D_{m+n}$, $\psi_{m,n}(D_m)\to 0$ on $D_m$ as $n \to \infty$.  Then  by Vitali's Theorem \cite[Theorem 3.3.3]{beardon}, $\psi_{m,n}(z)\to 0$ locally uniformly on the whole of $T^{-m}(U_m)$.
	
	Aiming for a contradiction, assume $T^{-m}(U_m)$ contains a point $z$ outside $C_m$.  Because $T^{-m}(U_m)$ contains $D_m$ and is connected we can construct a path $\gamma_0$, lying entirely inside $T^{-m}(U_m)$, joining $z$ to a point $z_0 \in D_m$.  Since $D_m$ lies inside $C_m$, $\gamma_0$ must cross $C_m$.  Truncate $\gamma_0$ to obtain $\gamma_0'$ joining $z_0$ to a point on $C_m$.  
	
	Apply $h_{m}$ to $\gamma_0'$ to obtain a new path, $\gamma_1$.  One end is $z_1 = h(z_0) \in D_{m+1}$.  The other lies outside $C_m$ by Lemma \ref{L:DISC show that the h_ns expand a circle}.  As before $\gamma_1$ must cross $C_m$ and so truncate it to obtain $\gamma_1'$ joining $z_1 \in D_{m+1}$ to a point on $C_m$.  Repeat, applying $h_{m+1}, h_{m+2}, \ldots, $ in turn, truncating each path to give a sequence of paths $( \gamma_n')$ and points $(z_n)$ where each $ z_n \in D_{m+n}$ ($n \geq 1$) and $\gamma_n'$ joins $z_n$ to a point on $C_m$.  By construction each $\gamma_n'$ is part of the image of the compact set $\gamma_0$ under the composition $\psi_{m,n} = h_{m+n-1}\circ \cdots \circ h_m = \psi_{m,n} $.  But as $n \to \infty $, $z_n \to 0 $ whereas each image has a point on the fixed circle $C_m$, which contradicts the uniform convergence $\psi_{m,n}(\gamma_0)\to 0$ as $n \to \infty$.  
	
	Accordingly the assumption that $T^{-m}(U_m)$ has a point $z$ outside $C_m$ is false and $T^{-m}(U_m)$ must be contained inside $C_m$.  It follows $U_m $ lies inside $ T^m (C_m) $.  This argument can be applied for every $m \geq N_0$ to complete the proof.
\end{proof}
We have now established $f$ has simply connected wandering domains ${\{U_n: n \geq 0\}}$ such that $f(U_n) \subset U_{n+1}$.  Moreover, for $n \geq N_0$, each $U_n$ contains the open disc $\D_n$ and lies inside the translated circle $T^n(C_n)$.
%==================================================================================================
% SECTION																						   
%==================================================================================================
%==================================================================================================
\section{The cauliflower and Hausdorff metric}\label{S:CAULI}
We will shortly consider how the shape of $U_n$ behaves as $n \to \infty$.  First, however, with that objective in mind, we introduce the cauliflower.  It is derived from the Fatou set of $q(z) = z-\pi z^2$. This quadratic is the only one (up to conjugation) with a parabolic basin, having a fixed point at $z=0$ with multiplier of $1$.  The properties of its Fatou set are known and we state a selection in the next lemma (see for instance \cite[Example 6.5.3]{beardon}, \cite[Para 173]{steinmetz} or \cite[pp. 97,115, 130]{carleson}).  
\begin{lemma}\label{L:CAULI properties of the quadratic}
	The quadratic $q(z)= z - \pi z^2$ has the following properties:
	\begin{enumerate}[(1)]
		\item The Fatou set of $q$ has exactly two simply connected components $W_0$, bounded, and $W_\infty$, unbounded. The Julia set $J(q)$ is their common boundary and contains $0$.  
		\item $\overline{W_0}$ lies inside the circle $\big\{\abs{z} = \tfrac{2}{\pi}\big\}$ (this is $C_1$ of Definition \ref{D:DISC define C_n, T, h_n and psi_m,n, w_n and H_a}).  
		\item\label{I:L:CAULI existence of rho} $W_0$ contains an open disc centred on the real axis and near $0$, for example $\disc{\tfrac{1}{6\pi}, \tfrac{1}{6\pi}}$.
		\item Both $q^n$ and $\arg q^n \to 0$ as $n \to \infty$ locally uniformly on $W_0$.
		\item\label{I:L:CAULI properties of q convergence of W_infty}As $n\to\infty$, $q^n \to \infty$ locally uniformly on $W_\infty$.
	\end{enumerate}
\end{lemma}
\begin{figure}[H]
	\centering
	\includegraphics[width=0.5\textwidth]{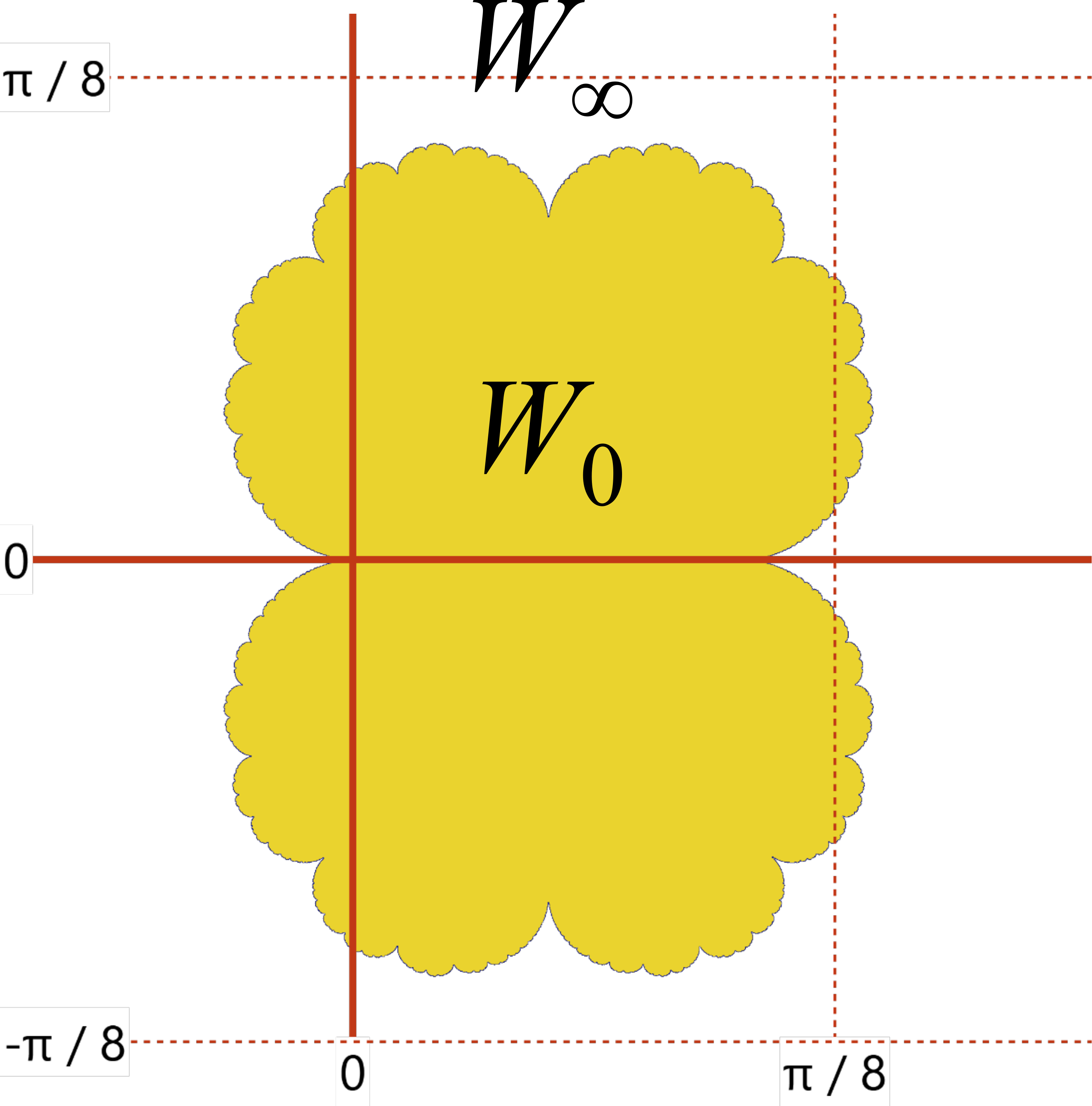}
	\caption{The Cauliflower}\label{F:CAULI image of cauliflower}
\end{figure}
The closure of the Fatou component $W_0$ is known as the cauliflower due to its distinctive appearance.  The author does not know the exact origin of the name but understands that the set was referred to as the \emph{chou-fleur} by Adrien Douady in the early 1980s; the name cauliflower appears in many places including \cite[p. 120]{MR2193309} and John Milnor’s 1990 notes for this text.  Since $W_0$ is simply connected and bounded by a Jordan curve (also in \cite{MR2193309}), its closure is the same as its filled in boundary.  

We will show that the Hausdorff distance between $\cauli$ and the sets $\overline{U_n}$, with the latter suitably scaled and re-positioned, converges to zero. 

A review of key properties of the Hausdorff metric that we need follows (see, for example, \cite[Page 134 f.]{falconer2014}).

\begin{definition} \label{D:CAULI Hausdorff metric}
	For a metric space $H$ with distance $\d{\cdot,\cdot}$ and subsets $X,Y \subset H$, the Hausdorff distance between $X$ and $Y$ is
		\begin{align*}
		    \dh{X, Y} = \max\left\{ \sup_{x\in X}\big \{ \inf_{y\in Y}\{\d{x,y}\} \big\},
		    	\sup_{y\in Y}\big \{ \inf_{x\in X}\{\d{x,y}\} \big\}
		    \right\}.
		\end{align*}
\end{definition} 
\begin{remark}	
The Hausdorff distance is a measure of closeness for any two subsets $X$ and $Y$.  If we restrict consideration to the set of non-empty compact subsets of $H$, then these form a metric space under $\dh{}$.  

When $X \subset Y$, $\inf_{y\in Y} \d{x,y} = 0$ for all $x \in X$, only the second term in Definition \ref{D:CAULI Hausdorff metric} applies, and $\dh{X,Y} = \sup_{y\in Y} \big\{\inf_{x\in X} \{\d{x,y}\} \big\}$.  It then follows that when $X_1 \subset X_2 \subset X_3 \subset H$ 
	\begin{align}\label{E:CAULI Hausdorff inclusion inequality}
		    \dh{X_1, X_2} \leq \dh{X_1, X_3} \text{  and  } \dh{X_2,X_3} \leq \dh{X_1, X_3}.
	\end{align}
\end{remark}
Additionally, for $\Omega_1, \Omega_2 \subset H$
	\begin{align}\label{E:CAULI Hausdorff distance same under closure}
	    \dh{\Omega_1, \Omega_2}
	    = \dh{\overline {\Omega_1}, \Omega_2 }
	    = \dh{\Omega_1, \overline{\Omega_2} }
	    = \dh{\overline {\Omega_1}, \overline{\Omega_2}}.
	\end{align}
We use the Hausdorff distance to create internal and external approximations to $\cauli$.
\begin{definition}
	For positive $\delta < \tfrac{1}{6\pi}$, thus small enough to ensure $\disc{\tfrac{1}{6\pi},\delta} \in W_0$, define:
	\begin{enumerate}[(1)] 
	\item $I_\delta = \bigcup \{\Omega \in \mathscr I_\delta\}$ where $\mathscr I_\delta$ is the collection of open and connected sets $\Omega$ such that $\tfrac{1}{6\pi} \in \Omega$ and for all $z \in \Omega$ the disc $\disc {z,\delta} \subset W_0$.  
	\item $E_\delta = \bigcup \{ \disc {z, \delta} : z \in W_0 \}$.
	\end{enumerate}
\end{definition}
\begin{remark}
	The collection $\mathscr I_\delta$ is constructed only for $\delta < \tfrac{1}{6\pi}$ to ensure it is non-empty and each member contains $\tfrac{1}{6\pi}$.  It follows $I_\delta$ is non-empty and connected.  The sets $E_\delta$ are non-empty and connected for any positive $\delta$ and no restriction is needed, but using the same requirement causes no difficulty and ensures the sets $E_\delta$ are uniformly bounded. 
\end{remark}
\begin{lemma}\label{L:CAULI interior and exterior set properties}
Let $\delta \in \big(0,\tfrac{1}{6\pi}\big)$. Then:
\begin{enumerate}[(1)]
\item Both $I_\delta$ and $E_\delta$ are non-empty open connected sets with $I_\delta \subset W_0$ and $\cauli \subset E_\delta$.\label{I:L:CAULI interior and exterior set properties I E non empty connected}
\item Both $I_\delta$ and $E_\delta$ are bounded and their closures are compact with $\overline{I_\delta} \subset W_0$.\label{I:L:CAULI interior and exterior set properties compactness I and E}
\item If $0 < \delta' < \delta$, $I_{\delta} \subset I_{\delta'}$ and $E_{\delta'} \subset E_\delta$.\label{I:L:CAULI interior and exterior set properties inclusion as delta shrinks}
\item $\displaystyle \bigcup_{\delta} I_\delta = W_0$ and $\displaystyle \bigcap_{\delta} E_\delta = \cauli$.\label{I:L:CAULI interior and exterior set properties union properties}
\item As $\delta \to 0$, $\dh{\overline{E_\delta}, \cauli} \to 0$, $\dh{\overline{I_\delta}, \cauli} \to 0$ and $\dh{\overline{E_\delta},\overline{I_{\delta}}} \to 0 $.\label{I:L:CAULI interior and exterior set H convergence}
\end{enumerate}
\end{lemma}
\begin{proof} 
\ref{I:L:CAULI interior and exterior set properties I E non empty connected}, \ref{I:L:CAULI interior and exterior set properties compactness I and E} and  \ref{I:L:CAULI interior and exterior set properties inclusion as delta shrinks} are elementary using the facts that $\delta < \tfrac{1}{6\pi}$ so that $\disc{\tfrac{1}{6\pi}, \delta} \in \mathscr I_\delta $ and $W_0$ is connected.

\ref{I:L:CAULI interior and exterior set properties union properties}.  For the first part, every point $z$ in $W_0$ may be joined to $\frac{1}{6\pi}$ by a path $\gamma$ in $W_0$.  Such a path is compact and so has a positive distance, $\varepsilon$ say, to $\partial W_0$. Create an open cover for $\gamma$ in $W_0$ consisting of discs of radius $\delta = \min\{\tfrac{1}{6\pi}, \frac{1}{2}\varepsilon\}$.  Their union meets the criteria for $\mathscr I_\delta$ and therefore $z \in \bigcup_\delta I_\delta$.  We already know the reverse inclusion.  For the second part, we know $\cauli \subset \bigcup_\delta E_\delta$ for all $\delta$.  In the opposite direction, $\bigcap_\delta E_\delta$ clearly contains all limit  points of $W_0$ and is therefore is a subset of $\cauli$.

\ref{I:L:CAULI interior and exterior set H convergence}. Fix $\delta$ and consider $E_\delta$ first. It contains $\cauli$ and the Hausdorff distance reduces to
	\begin{align*}
	    \dh{\cauli, \overline{E_\delta}} = \sup_{\zeta \in \overline{E_\delta}} \big\{ \inf_{z \in \cauli}\{ \d{\zeta,z} \} \big\}.
	\end{align*}
Take any $\zeta \in \overline{E_\delta}$.  Then there exist a sequence $\zeta_k \in E_\delta$ with $\zeta_k \to \zeta $ as $k \to \infty$.  For $\varepsilon > 0$, choose~$ k $ so that $ \d {\zeta_k,\zeta} < \varepsilon$.  There exists $z_k\in W_0$ such that $\zeta_k \in \disc {z_k,\delta}$ and the triangle inequality implies that $\d{\zeta,z_k} < \delta + \varepsilon$.  It now follows that for every $\zeta \in \overline {E_\delta}$
	\begin{align*}
	    \inf_{z \in \cauli} \{ \d{\zeta, z} \} 
	    	&= \inf_{z \in W_0} \{ \d{\zeta,z} \} \\
			&\leq \d{\zeta, z_k} \\
			&\leq \delta + \varepsilon.
	\end{align*}
It follows that $\dh{\cauli,\overline{E_\delta}} \leq \delta + \varepsilon$.  We chose $\varepsilon$ arbitrarily and the result follows.

Now consider $I_\delta$.  Chose $\varepsilon >0$.  Since $\overline{I_{\delta}} \subset \cauli$, as above the Hausdorff distance simplifies to 
	\begin{align*}
	    \dh {\overline{I_{\delta}}, \cauli} = \sup_{\zeta \in \cauli} \big\{ 
	    	\inf_{z \in \overline{I_{\delta}}} \{ \d {\zeta, z} \}\big\}
	    	= \sup_{\zeta \in \cauli} \big\{ \inf_{z \in I_{\delta}} \{ \d {\zeta, z} \}
	    \big\}.
	\end{align*}
Accordingly, for $ n \geq 1$ we can find $\zeta_n \in \cauli$ satisfying
 	\begin{align}\label{E:CAULI lower bounds on infimum of d(zeta_n, z)}
 	    \inf\{\d{\zeta_n, z} : z \in I_{1/n} \} + \frac{1}{n} > \dh {\cauli, I_{1/n}}.
 	\end{align}
The remarks after Definition \ref{D:CAULI Hausdorff metric} imply $\dh {\cauli, \overline{I_\delta}}$ decreases with $\delta$, and so has a limit $\mu \geq 0$ as $\delta \to 0$.  

The $(\zeta_n)$ form a sequence in the compact set $\cauli$ which therefore has a convergent subsequence $(\zeta_{n_k})$ with limit $\zeta \in \cauli$.  Since $\zeta \in \cauli$ there exists $\zeta' \in W_0$ such that $\d{\zeta,\zeta'} < \varepsilon$.  By part \ref{I:L:CAULI interior and exterior set properties union properties}, we can find $K_1$ such that $\zeta' \in I_{1/{n_k}}$ for all $k \geq K_1$; we can find $K_2$ such that $\d{\zeta,\zeta_{n_k}} < \varepsilon$ whenever $k \geq K_2$; and we can find $K_3$ such that $1/{n_k} < \varepsilon$  whenever $k \geq K_3$.  Then, for $k \geq \max\{ K_1,K_2,K_3\}$,
	\begin{align*}
		\mu &\leq \dh {\cauli,I_{1/n_k} } \\
	    &< \inf\{\d{\zeta_{n_k}, z}: z \in I_{1/n_k}\}  + \varepsilon \\
	    &< \d {\zeta_{n_k},\zeta}+\d{\zeta,\zeta'} + \varepsilon \\
	    &< 3\varepsilon.
	\end{align*}
Since $\varepsilon>0$ is chosen freely, $\mu=0$.   The remainder of the proof follows from the triangle inequality.  
\end{proof}
%==================================================================================================
% SECTION																						   
%==================================================================================================
%==================================================================================================
\section{The convergence of the wandering domains' shape}\label{S:CONV}
To establish that the shape of $U_n$ approximates the cauliflower, it needs to be scaled and re-positioned.  We therefore introduce versions of the functions $h_n$ and $\psi_{m,n}$ of Definition \ref{D:DISC define C_n, T, h_n and psi_m,n, w_n and H_a} as follows.
\begin{definition}\label{D:CONV definitions of phi_m,n and g_m}
Define the following:
\begin{enumerate}[(a)~]
	\item for integers $n \geq 1$,
		\begin{align*}
			g_n(z) = (n+1)h_n(z/n) = (n+1)\left( \left(\tfrac{z}{n}+ 2n\pi\right)\cos\left(\tfrac{z}{n}\right) - 2n\pi \right);
		\end{align*}
	\item
	for $m \geq 1,n \geq 0$, 
		\begin{align*}
		    \phi_{m,n}(z) = (m+n) \psi_{m,n}(z/m);
		\end{align*}
	\item
	for $m \geq 1 $ 
		\begin{align*}
   			V_m = m T^{-m}(U_m).
		\end{align*}
	\end{enumerate}
\end{definition}
\begin{remark}
	It follows from the definitions that for $ m \geq 1, n \geq 0$
		\begin{align*}
		    \phi_{m,n}(z) 
		    = g_{m+n-1} \circ g_{m+n-2} \circ\cdots\circ g_m(z) 
		    = (m+n) T^{-(m+n)} \circ f^n \circ T^m(z/m).
		\end{align*}
	We can express $g_n$ as a power series as we did for $h_n$ in equation \eqref{E:DISC: power series for h_n}.
		\begin{align}\label{E:CONV power series for g_n}
			g_n(z)
				&= \frac{n+1}{n}\sum_{k=0}^{\infty} \frac{(-1)^k}{n^{2k}} 
					\left(\frac{z^{2k+1}}{(2k)!} - \frac{2\pi z^{2k+2}}{(2k+2)!}\right)
				\\&= \frac{n+1}{n}\left(z - \pi z^2\right) 
					- \frac{n+1}{2! n^3} z^3 
				\nonumber
				\\&\hspace{1cm}
					+ \frac{2(n+1)\pi}{4! n^3} z^4 
					+ \frac{n+1}{4! n^5}z^5 
				\nonumber
				\\&\hspace{2cm}
					- \frac{2(n+1)\pi}{6! n^5} z^6 
					- \frac{n+1}{6! n^6} z^6
					+\cdots .
				\nonumber
		\end{align}
	The sets $V_m$ are simply the $U_m$ translated to the origin and scaled by a factor $m$.  It is these that converge to the cauliflower.  Lemma \ref{L:WANDER U_n are bounded} implies for $m \geq N_0$ that $V_m$ is contained in the circle $\{z: \abs{z} = \tfrac{2}{\pi}\}$.
\end{remark}
Recall that in Lemma \ref{L:CAULI properties of the quadratic} we defined $q(z)=z-\pi z^2$.  The next steps show that for $n \geq 0$ and large enough $m$ the functions $\phi_{m,n}$ provide uniform approximation for $q^n$ on a suitable domain.  The first step towards this is to show that $g_n \to q$ as $n \to \infty$.
\begin{lemma}\label{L:CONV uniform convergence of g_m to q}
	Let $r > 0$.  Then there exists $\mu(r) > 0$ dependent only on $r$ such that
		\begin{align*}
		    \abs{g_m(z) - q(z) } \leq \frac{\mu(r)}{m}
		\end{align*}
	for all $m \geq 1$ and all $z \in \overline {\disc{0,r}}$.  This implies $g_m(s) \to q(z)$ uniformly on $\overline {\disc{0,r}}$ as $m \to \infty$.
\end{lemma}
\begin{proof}
	Let $z \in \overline{ \disc{0,r}}$.  For $m \geq 1$ we use the power series expansion \eqref{E:CONV power series for g_n}
		\begin{align*}
		    \abs{g_m(z) - q(z)} &= \frac{1}{m}\abs{z -\pi z^2 - \frac{(m+1)z^3}{2!m^2} + \frac{2 \pi (m+1)z^4}{4!m^2} + \frac{(m+1)z^5}{4!m^4} - \cdots } \\
		    &\leq \frac{1}{m}\abs{\abs{z} + \pi \abs{z}^2 + \frac{2\abs{z}^3}{2!} + \frac{4 \pi \abs{z}^4}{4!} + \frac{2\abs{z}^5}{4!} + \cdots } .
		\end{align*}
	The series sums to a continuous function of $\abs{z}$ which is bounded on compact sets, in particular on $\overline {\disc {0,r}}$.  The bound depends on $r$ but not on $m$. Denoting the bound by $\mu(r)$ we can write
		\begin{align*}
		    \abs{g_m(z) - q(z)} &\leq \frac{\mu(r)}{m}.  \qedhere
		\end{align*}
\end{proof}
The second lemma deals with the continuity of $q$ and $g_n$.
\begin{lemma}\label{L:CONV uniform continuity of g_m and q}
Let $r > 0$ and $m \geq 1$.  Then both $g_m$ and $q$ are uniformly continuous on $\overline {\disc {0,r}}$.
 
Moreover, for any $\varepsilon>0$ there exist $\delta(\varepsilon, r) >0$ and $M(\varepsilon, r) \geq 1$ such that for all $m \geq M(\varepsilon, r)$
 	\begin{align*}
 	    \abs{g_m(z_1)-g_m(z_2)} < \varepsilon
 	\end{align*}
whenever $z_1, z_2 \in \overline{ \disc {0,r} }$ with $\abs{z_1-z_2} < \delta(\varepsilon, r)$.
\end{lemma}
\begin{proof}
	Fix $r$.  Since $g_m$ and $q$ are continuous in $\Complex$ they are uniformly continuous on compact subsets.
	
	For uniformity with respect to $m$, let $\varepsilon > 0$ be given.  By Lemma \ref{L:CONV uniform convergence of g_m to q}, we can find $M(\varepsilon, r)$ so that $\frac{\mu(r)}{M(\varepsilon, r)} < \frac{1}{3}\varepsilon$.  
	
	By uniform continuity of $q$ we can also find $\delta(\varepsilon, r)$ such that $\abs{q(z_1) - q(z_2)}< \frac{1}{3}\varepsilon$ whenever $z_1, z_2 \in \overline {\disc{0,r}}$ with $\abs{z_1-z_2}<\delta(\varepsilon, r)$.
	
	Then, for such $z_1, z_2$ and $m \geq M(\varepsilon, r)$
		\begin{align*}
		    \abs{g_m(z_1)-g_m(z_2)} 
		    	&= \abs{g_m(z_1) - q(z_1) + q(z_1)-q(z_2) + q(z_2)-g_m(z_2)} \\
		    	&\leq \abs{g_m(z_1) - q(z_1)} + \abs{q(z_1)-q(z_2)} + \abs{q(z_2)-g_m(z_2)} \\
		    	&<\frac{1}{3}\varepsilon + \frac{1}{3}\varepsilon + \dfrac{1}{3} \varepsilon. \qedhere
		\end{align*}
\end{proof}
We now establish the main convergence property for the family $\{\phi_{m,n}\}$.
\begin{lemma}\label{L:CONV convergence of phi_m,n to q^n}
	For $n \geq 0$, $r >0$ and $\varepsilon > 0$ there exists $M(\varepsilon, r, n) \geq 1$ such that
		\begin{align}\label{E:CONV main inequality q^n - phi_m,n < varepsilon}
		    \abs{q^n(z)-\phi_{m,n}(z)} < \varepsilon
		\end{align}
	for all $m \geq M(\varepsilon, r, n)$ and all $z \in \overline{ \disc {0,r}}$.
\end{lemma}
\begin{proof}
In the following proof, for $r > 0$ and $k \geq 1$, let $L(r,k) = \tfrac{1}{5}\max\big\{ 5^{2^k}, (5r)^{2^k}\big\}$.  For all $z$, by considering $\abs{z}< 1$ and $\abs{z}\geq 1$ separately, we can see that $\displaystyle \abs{q(z)}< 5 \max\{ 1,\abs{z}^2\}$.  Then, for any $r >0$, when $\abs{z} < r$, because $5 L(r, k)^2 = L(r,k+1) $, we can obtain by induction
	\begin{align}\label{E:CONV upper bound on q^k}
	    \abs{q^k(z) } < L(r,k).
	\end{align}
We prove the inequality \eqref{E:CONV main inequality q^n - phi_m,n < varepsilon} by induction on $n$.  When $n=0$, $q^n$ and $\phi_{m,n}$ are both the identity functions and \eqref{E:CONV main inequality q^n - phi_m,n < varepsilon} is trivially true on every closed disc and for all $m \geq 1$.

Suppose that the lemma and equation \eqref{E:CONV main inequality q^n - phi_m,n < varepsilon} have been established for all $r > 0$, $\varepsilon > 0$ when $n=k$, for some $k \geq 0$.  

Let $\varepsilon > 0$ and $r > 0$ be given.

It follows from \eqref{E:CONV main inequality q^n - phi_m,n < varepsilon} when $n = k$ that we find $M_1(1, r, k)$ such that 
	\begin{align*}
	    \abs{q^k(z) - \phi_{m,k}(z)} < 1
	\end{align*}
for all $z \in \overline {\disc{0,r}}$ and $m \geq M_1(1, r, k)$.  Combine this with \eqref{E:CONV upper bound on q^k} to obtain 
	\begin{align}\label{E:CONV phi_m,k in B(0,r_1)}
	    \abs{\phi_{m,k}(z)} < 1 + L(r,k).
	\end{align}
Write $R(r,k) = 1+L(r,k)$ and use Lemma \ref{L:CONV uniform convergence of g_m to q} to obtain $\mu(R(r,k)) >0$ such that for $z \in \overline {\disc{0, R(r,k)}}$
	\begin{align*}
	    \abs{g_m(z)-q(z)} < \frac{\mu(R (r,k))}{m}.
	\end{align*}
For the given $\varepsilon$ we can use this to find $M_2(\varepsilon, r,k)$ such that
	\begin{align}\label{E:CONV g_m is close to q for large m}
	    \abs{g_m(z)-q(z)} < \tfrac{1}{2}\varepsilon
	\end{align}
for all $m \geq M_2(\varepsilon, r,k)$ and $z \in \overline {\disc{0,R (r,k)}}$.

We now use the uniform continuity of $q$ and Lemma \ref{L:CONV uniform continuity of g_m and q} to find $M_3(\varepsilon, r,k)$ and $\delta(\varepsilon, r,k)$ such that for $m \geq M_3(\varepsilon, r,k)$ and $z_1, z_2 \in \overline{ \disc {0,R_{r,k}}}$ with ${\abs{z_1-z_2}<\delta(\varepsilon, r,k)}$
	\begin{align}\label{E:CONV apply uniform continuity to q and g_m}
	    \abs{q(z_1)-q(z_2)} < \tfrac{1}{2}\varepsilon \quad\text{and}\quad \abs{g_m(z_1)-g_m(z_2)} < \tfrac{1}{2}\varepsilon.
	\end{align}
We use \eqref{E:CONV main inequality q^n - phi_m,n < varepsilon} again, with $n=k$ and $r > 0$ to obtain $M = M(\delta(\varepsilon, r, k), r, k)$ such that for $m \geq M$ and $z \in \overline {\disc {0,r}}$
	\begin{align}\label{E:CONV phi and q within delta of eachother}
		    \abs{q^k(z) - \phi_{m,k}(z)} < \delta(\varepsilon, r,k).
	\end{align}
This $M$ depends indirectly on the given $\varepsilon$ and so we denote it $M_4(\varepsilon, r, k)$.  

By the triangle inequality, for any $z$,
	\begin{align*}
	    \abs{q^{k+1}(z) - \phi_{m,k+1}(z)} 
	    	&= \abs{q\left(q^k(z)\right) - g_{m+k}\left(\phi_{m,k}(z)\right)} 
	    	\\&
	    	\leq \abs{q\left(q^k(z)\right) - g_{m+k}\left(q^k(z)\right)} 
	    	\\&\hspace{1cm}
	    	+ \abs{g_{m+k}\left(q^k(z)\right)-g_{m+k}\left(\phi_{m,k}(z)\right)}.
	\end{align*}
Let $M(\varepsilon, r,k)=\max \{ M_1(1, r,k), M_2(\varepsilon, r,k), M_3(\varepsilon, r,k), M_4(\varepsilon, r,k) \}$. 

Then for ${m \geq M(\varepsilon, r,k)}$ and $z \in \overline {\disc {0,r}}$, inequalities \eqref{E:CONV upper bound on q^k} and \eqref{E:CONV phi_m,k in B(0,r_1)} imply both $q^k(z), \phi_{m,k}(z) \in \overline {\disc {0,R(r,k)} }$.  Because $m+k \geq M_2(\varepsilon, r,k)$, inequality \eqref{E:CONV g_m is close to q for large m} implies the first term on the right is less than $\tfrac{1}{2}\varepsilon$.  Next, $m \geq M_4(\varepsilon, r,k) $ so that inequality \eqref{E:CONV phi and q within delta of eachother} holds and since $m+k \geq M_3(\varepsilon, r,k)$ the second inequality in \eqref{E:CONV apply uniform continuity to q and g_m} may be used, giving
	\begin{align*}
	    \abs{g_{m+k}\big(q^k(z)\big) - g_{m+k}\big(\phi_{m,k}(z)\big)} < \tfrac{1}{2}\varepsilon.
	\end{align*}
Thus the hypothesis and \eqref{E:CONV main inequality q^n - phi_m,n < varepsilon} extend to $n=k+1$ for given $r$ and $\varepsilon$.  These quantities were arbitrary and therefore by induction the lemma holds for all $n\geq 1$.
\end{proof}
Recalling Definition \ref{D:CONV definitions of phi_m,n and g_m}, we are ready to show that for $\delta$ small enough to ensure $I_\delta$ is non-empty (that is, $\disc{\tfrac{1}{6\pi},\delta}$ is inside $W_0$), the sets $\{V_n\}$ eventually lie between $I_\delta$ and $E_\delta$. 
\begin{lemma}\label{L:CONV V_n between interior and exterior sets for large n}
	For any $\delta \in \big(0,\frac{1}{6\pi}\big)$, there exists $M(\delta)\geq 1$ such that for all $m \geq M(\delta)$, $I_\delta \subset V_m \subset E_\delta$.
\end{lemma}
\begin{proof}
	We have shown $ \overline{I_\delta}$ is a compact subset of $W_0$ (Lemma \ref{L:CAULI interior and exterior set properties}) so Lemma \ref{L:CAULI properties of the quadratic} may be applied and as $n \to \infty$ both $q^n$ and $\arg q^n$ converge uniformly to zero on $\overline{I_\delta}$.  Therefore we can find $N$ such that for $n\geq N$, $K = q^n(\overline{I_\delta})$ is a compact subset of the open disc $D_1$.  As $K$ is compact, there is a positive distance between $\partial D_1$ and $K$.
	
	Fix $n\geq N$, and choose a disc that contains both $W_0$ and $\overline{I_\delta}$.  Use Lemma \ref{L:CONV convergence of phi_m,n to q^n} to find $M$ such that on this disc for all $m \geq M$, $\phi_{m,n}$ is close enough to $q^n$ that $\phi_{m,n}(\overline{I_\delta})$ also lies inside $D_1$.  We can suppose without loss of generality that $M \geq N_0$.
	
	Then, for $m\geq M$, it follows from the definition of $\phi_{m,n}$ that
		\begin{align*}
		    f^n\circ T^m\left(\tfrac{1}{m}\overline {I_\delta} \right) \subset T^{m+n} \left(\tfrac{1}{m+n} D_1\right) = \D_{m+n}.
		\end{align*}
	We showed in Lemma \ref{C:DISC psi and fn normal on appropriat D_m} that for $m \geq N_0$, $\D_{m+n}$ is in $F(f)$ and as the Fatou set is completely invariant $T^m\left(\frac{1}{m}\overline {I_\delta} \right)$ is also in the Fatou set.  But $I_\delta$ contains the point $\frac{1}{6\pi}$ and therefore $T^m\left(\frac{1}{m}\overline{I_\delta}\right)$ contains the point $\frac{1}{6m\pi}+2m\pi$ which is in $\D_m$. The component of the Fatou set containing $\D_m$ is $U_m$ and so 
		\begin{align*}
		    T^m\left(\tfrac{1}{m}\overline{I_\delta}\right) \subset U_m
		\end{align*}
	whence, using the definition of $V_m$, $\overline{I_\delta} \subset V_m$ as required.
	
	For the second inclusion, take $\delta > 0$ and let $R > \frac{2}{\pi} $ be large enough that $E_\delta, W_0$ and $V_n \subset \disc {0,R}$ for $n \geq N_0$.  Let $K = \{ z \in \Complex \setminus E_\delta: \abs{z} \leq R \}$.  Then $K$ is compact.  Moreover $K \subset W_\infty$ so that by Lemma \ref{L:CAULI properties of the quadratic}, item \ref{I:L:CAULI properties of q convergence of W_infty} as $n \to \infty$, $q^n\to\infty$ uniformly on $K$ and we can find $n$ such that $\abs{q^n(z)} > R + 2$ for all $z \in K$.  For this $n$, using Lemma \ref{L:CONV convergence of phi_m,n to q^n}, we can find $M$ such that for all $m \geq M$, $\abs{\phi_{m,n}-q^n} < 1$ on $K$.  
	
	Now suppose that there is an infinite number of the sets $\{V_m\}$ that contain points outside $E_\delta$.  Then we can find one such with $m \geq M$ and, because $V_m$ is connected and contains points close to zero, this point connects to near zero by a path in $V_m$ that must cross the boundary of $E_\delta$.  Accordingly we can find $z \in V_m$ lying outside $E_\delta$ but arbitrarily close to its boundary.  Such $z$  is inside $V_m \cap K$.  Then
		\begin{align*}
		    \abs{\phi_{m,n}(z)} > \abs{q^n(z)} - \abs{\phi_{m,n}(z)-q^n(z)} > R + 2 - 1 > R.
		\end{align*}
	But $\phi_{m,n}(V_m) \subset V_{m+n}$ and $V_{m+n}$ is within $\disc{0,R}$, which gives a contradiction.  Therefore for large enough $m$ all $V_m$ must lie within $E_\delta$.  	
\end{proof}
The convergence in Hausdorff metric of $V_n$ to $\cauli$ follows easily.
	\begin{lemma}\label{L:CONV shapes V_n converge to the cauliflower}
		As $n \to \infty $, $\dh {\overline{V_n}, \cauli} \to 0$.
	\end{lemma}
	\begin{proof}  
		We have seen that $\dh {\overline{I_\delta}, \overline{E_\delta}} \to 0$ as $\delta \to 0$ and for large $m$, $I_\delta \subset V_m \subset E_\delta$.  Use the inequalities in \eqref{E:CAULI Hausdorff inclusion inequality} and Lemma \ref{L:CONV V_n between interior and exterior sets for large n} to trap both $\dh {\overline{I_\delta}, \cauli}$ and $\dh {\overline{I_\delta}, \overline{V_m}}$ below $\dh {\overline{I_\delta}, \overline{E_\delta}} $.  Then the triangle inequality for $\dh{}$ means $\dh {\overline{V_m}, \cauli} < 2\dh {\overline{I_\delta}, \overline{E_\delta}} \to 0$.
	\end{proof}
\section{Classification of the wandering domains}\label{S:DYNA}
Our final objective is to classify the wandering domain $U_0$ within the ninefold framework set out in \cite{benini+2021}.  That paper's first classification \cite[Theroem A]{benini+2021} implies that for a wandering domain $U$ and its orbit $(U_n)$, the hyperbolic distance between $f^n(x), f^n(y)$ with respect to $U_n$ follows exactly one of three rules: (1) contracting:  the distance has limit zero for all pairs;  (2) semi-contracting: the distance has a positive limit which is never attained (not necessarily the same for all pairs) except on the set $E$ of pairs  $(z,z')$ for which $f^k(z) = f^k(z')$ for some $k \in \Natural$; or (3) eventually isometric: the distance becomes constant (not necessarily the same for all pairs) except on $E$.  Our contention, proved below, is that in our case the wandering domain $U_0$ is contracting.  

The second classification theorem \cite[Theorem C]{benini+2021} states that iterates $f^n(z)$ of $z \in U$ behave in exactly one of three ways:  (a) the iterates $f^n(z)$ stay away from $\partial U_n$; (b) among the iterates, one subsequence approaches $\partial U_n$ and another remains away from it; or (c) all iterates converge to $\partial U_n$.  The stated version of the theorem uses the Euclidean distance.  In our case, since the components $(U_n)$ of our wandering domain shrink in size to zero, it is clear that all iterates converge to the boundary and our wandering domain is of type (c) under this classification.

To show that our wandering domain is contracting some preliminaries are needed.  We return to Definition \ref{D:DISC define C_n, T, h_n and psi_m,n, w_n and H_a} and the commutative diagram in Figure \ref{F:DISC commutative diagram}.  We consider a real point $t_0 \in H_{(3m+1)\pi}$ for fixed $m \geq N_0$.
\begin{lemma}\label{L:DYNA distance between iterates in H_m}
	Let $m \geq N_0$ and take real $t_0 \in H_{3(m+1)\pi}$.  
	
	Define for $n \geq 1$, $t_n = w_{m+n-1}\circ w_{m+n-2}\circ\cdots\circ w_m(t_0)$.  Then
	\begin{enumerate}[(1)]
	\item\label{I:L:DYNA Distance between iterates t_n points are distinct}$t_{n+1} > t_n$;
	\item\label{I:L:DYNA Distance between iterates quadratic growth of t_n} $t_n \in H_{b_{m,n}}$ where $b_{m,n} = 3(m+1)\pi + \frac{9\pi}{20}\left( m + (m+1) + \cdots + (m+n-1)\right)$;
	\item\label{I:L:DYNA Distance between iterates linear progression of H_{3n pi}} $b_{m,n} > 3(m+n+1)\pi$;
	\item\label{I:L:DYNA Distance between iterates hyperbolic distance t_n to t_n+1 in H_m}for each fixed $m$, the hyperbolic distance $\d[H_{3(m+n+1)\pi}] {t_n, t_{n+1}} \to 0$ as ${n \to \infty}$.
	\end{enumerate}
\end{lemma}
\begin{proof}
	Item \ref{I:L:DYNA Distance between iterates t_n points are distinct} follows directly from \eqref{E:DISC real part t inequality} of Lemma \ref{L:DISC w_n moves to right on half-plane}.  Repeated use of the inequality also gives the sum
		\begin{align}\label{E:DYNA distance of t_n to boundary}
		    \Re t_n 
		    &
		    > \Re t_0 + \frac{9\pi}{20}\left(m + (m+1)+\cdots + (m+n-1)\right) 
		    \nonumber
		    \\&
		    > 3(m+1)\pi + \frac{9\pi}{40}n(2m+n-1),
		\end{align}
	which leads to \ref{I:L:DYNA Distance between iterates quadratic growth of t_n} and \ref{I:L:DYNA Distance between iterates linear progression of H_{3n pi}} for $n \geq 1$ and $m \geq N_0 $ (using $N_0 = 7$).
	
	For \ref{I:L:DYNA Distance between iterates hyperbolic distance t_n to t_n+1 in H_m}, inequality \eqref{E:DISC real part t inequality} also provides the upper bound 
		\begin{align}\label{E:DYNA separation of t_n and t_n+1}
	    	t_{n+1} - t_n < \frac{27}{17}(m+n)\pi.  
		\end{align}
	In a half-plane $H_a$ the hyperbolic density is given by (see \cite[Example 7.2]{beardon2006} for instance),
		\begin{align*}
		    \lambda_{H_a}(z) = \frac{1}{\Re z - a}.
		\end{align*}
	Moreover, the real axis in $H_a$ is a geodesic for $H_a$ and therefore for any real points $x_1 < x_2 \in H_a$, we have
		\begin{align*}
		    \d[H_a] {x_1,x_2} = \int_{x_1}^{x_2} \frac{dt}{t-a} \leq \frac{x_2-x_1}{x_1-a}.
		\end{align*}
	We apply this to $t_n$ and $t_{n+1} \in H_{3(m+n+1)\pi}$.  From \eqref{E:DYNA separation of t_n and t_n+1}
		\begin{align*}
		    t_{n+1}-t_n < \dfrac{27\pi}{17}(m+n)
		\end{align*}
	and, from \eqref{E:DYNA distance of t_n to boundary}, when $m \geq N_0$ (recall $N_0 = 7$),
		\begin{align*}
		    t_n - 3(m+n+1)\pi >\frac{n\pi}{40}(18m + 9n - 9 - 120). 
		\end{align*}
	 so that 
		\begin{align*}
		    \d[H_{3(m+n+1)\pi}] {t_n, t_{n+1}} \leq \frac{\frac{27}{17}(m+n)}{\frac{1}{20} n(18m+9n-129)\pi} \to 0 
		\end{align*}
	as $n \to \infty $ with $m$ fixed, as required.	
\end{proof} 
We will also use elementary properties of $f$ on the real axis.  These are gathered here.  
\begin{lemma}\label{L:DYNA f is incresing on small sets}
	For $m \geq N_0$, the function $f$ is increasing on both $\D_m \cap \Real$ and $T(\D_m )\cap \Real$.  It follows that $f$ is a bijection from $\D_m \cap \Real $ onto $f(\D_m) \subset \D_{m+1} \cap \Real$.  
\end{lemma}
\begin{proof}
		Let $k = m+1$ or $m$.  Then, let $x \in I = T^{k}(D_m) \cap \Real = \left(2k\pi, 2k\pi+\frac{1}{6m\pi}\right)$ and write $x = 2k\pi + t$ so that $t \in \left(0, \frac{1}{6m\pi}\right)$. It follows that $\cos x = \cos t > 1-\frac{t^2}{2}$ and $\sin x = \sin t< t < \frac{1}{6m\pi}$, from which $f'(x)>0$ follows directly.
\end{proof}
\begin{lemma}\label{L:DYNA ordering of x_n y_n}
	Take fixed $m \geq N_0$ and define real sequences $\{x_n\}$ and $\{y_n\}$ for $n \geq 0$ as follows:  choose real $y_0 \in \D_m$. Define $x_0 = T^{-1}(f(y_0))$.  For $n \geq 1$ let $x_n = f^n(x_0)$ and  $y_n = f^n(y_0)$.
	
	Then the following hold:
	\begin{enumerate}[(1)]
		\item\label{I:L:DYNA ordering of x_n,y_n basic x_n y_n properties} for all $n \geq 0$ we have $x_n, y_n \in \D_{m+n} \cap \Real$;
		\item\label{I:L:DYNA ordering of x_n,y_n Tx_n/y_n bigger than x_n+1/y_n+1}for all $n \geq 0$, $T(x_n) > x_{n+1} > 2(m+n+1)\pi$ \\ and ${T(y_n) > y_{n+1} >  2(m+n+1)\pi}$;
		\item\label{I:L:DYNA ordering of x_n y_n, Tf > fT}for all $x \in \Real$, $f\circ T(x) \geq T\circ f(x)$;
		\item\label{I:L:DYNA ordering of x_n, y_n y_n > x_n}for $n \geq 0$, $y_n > x_n$;
		\item\label{I:L:DYNA T ordering of x_n, y_n T x_n >= y_n+1}for all $n \geq 0$, $T(x_n) \geq y_{n+1}$;
		\item\label{I:L:DYNA ordering of x_n, y_n x_n in D_m+n and D_m+n+1 - 2pi}for all $n \geq 0$ $x_n \in T^{-1}(\D_{m+n+1}) \cap \D_{m+n}$.
	\end{enumerate}
\end{lemma}
\begin{proof}  \ref{I:L:DYNA ordering of x_n,y_n basic x_n y_n properties}.  Since $y_0$ is real and in $\D_m$, it is plain from the definitions of $f$ and $T$ that $(x_n)$ and $(y_n)$ are real.  We also have seen that for all $k \geq N_0$, the function $f$ maps $\D_k$ into $\D_{k+1}$ from which it follows at once that $x_n, y_n \in \D_{m+n}$.

\ref{I:L:DYNA ordering of x_n,y_n Tx_n/y_n bigger than x_n+1/y_n+1}. Both $x_n$ and $y_n$ lie in $\D_{m+n} \cap \Real$.  Because $f(\D_{m+n}) \subset \D_{m+n+1}$, the right hand inequalities follow directly;  at the same time $\cos x_n$ and $\cos y_n$ are in $(0,1)$ because no $x_n$ or $y_n$ is a multiple of $2\pi$, so $x_{n+1} = x_n \cos x_n + 2\pi < T(x_n)$. The argument for $y_n$ is the same.

\ref{I:L:DYNA ordering of x_n y_n, Tf > fT}.  We have $ T\circ f(x) - f\circ T(x) = 2\pi (1 - \cos x) \geq 0$ directly from the definition of $f$.

\ref{I:L:DYNA ordering of x_n, y_n y_n > x_n}. When $n=0$, $f(y_0) = T(x_0)$ and $y_0$ is not a multiple of $2\pi$, so that
	\begin{align*}
	    y_0 - x_0 
	    	&= T(y_0) - T(x_0)
	    	\\& = T(y_0) - f(y_0)
	    	\\& = y_0(1-\cos y_0)
	    	\\&> 0.
	\end{align*}
The statement \ref{I:L:DYNA ordering of x_n, y_n y_n > x_n} now follows because for all $k \geq 0$, we have $f$ is strictly increasing on $\D_{m+k}$ (Lemma \ref{L:DYNA f is incresing on small sets}) and $f(\D_{m+k}) \subset \D_{m+k+1}$, so that $f^n$ is increasing on $\D_m$.

\ref{I:L:DYNA T ordering of x_n, y_n T x_n >= y_n+1}.  The proof is by induction.  $T(x_0) = y_{1}$ when $n=0$.  Suppose \ref{I:L:DYNA T ordering of x_n, y_n T x_n >= y_n+1} holds for some $k \geq 0$.  By definition, $T(x_{k+1}) = T\circ f(x_{k}) $ and $y_{k+2} = f^2(y_{k})$ so that
	\begin{align*}
	    T(x_{k+1}) - y_{k+2} 
	    	&		= x_k \cos x_k + 4\pi - \big((y_{k}\cos y_k  + 2\pi)\cos (y_k\cos y_k + 2\pi) + 2\pi\big)
	    	\\&		= x_k \cos x_k - (y_{k}\cos y_k)\cos (y_k\cos y_k) +2\pi\big(1 - \cos(y_k\cos y_k)\big).
	\end{align*}
Now, since \ref{I:L:DYNA T ordering of x_n, y_n T x_n >= y_n+1} holds when $n=k$, $x_k \geq  y_k \cos y_k$.  Moreover both $x_k, y_k \in \D_{m+k} \cap \Real $. Therefore $y_k \cos y_k$ is also in $\D_{m+k}$ because $\cos y_k \in (0,1)$ and so we can use the fact that $t\cos t$ is increasing and deduce that $x_k \cos x_k \geq  (y_k\cos y_k)\cos (y_k \cos y_k)$.  Therefore
	\begin{align*}
	    T(x_{k+1}) - y_{k+2} 
	    	&	\geq 2\pi(1 - \cos (y_k \cos y_k))
			> 0
	\end{align*}
where the inequality becomes strict since $y_k$ is not a multiple of $2\pi$ and $\cos y_k \in (0,1)$.  Thus \ref{I:L:DYNA T ordering of x_n, y_n T x_n >= y_n+1} holds for $n=k+1$ and so, by induction, for all $n \geq 0$.

\ref{I:L:DYNA ordering of x_n, y_n x_n in D_m+n and D_m+n+1 - 2pi}.  We already know by construction that $x_n \in \D_{m+n}$ for all $n\geq 0$ and we need only prove the stricter membership $T(x_n) \in \D_{m+n+1}$.  We use induction again.  

By construction, $ {y_1 \in f(\D_{m})\subset \D_{m+1} \subset T(\D_m) }$ and since $x_0 = T^{-1}(y_1)$ it follows $x_0 \in T^{-1} (\D_{m+1})$.  Thus \ref{I:L:DYNA ordering of x_n, y_n x_n in D_m+n and D_m+n+1 - 2pi} holds when $n=0$.  

If now for some $n\geq 0$ we have shown $x_n \in T^{-1}(\D_{m+n+1}) $, then ${T (x_n) \in \D_{m+n+1}}$ and $f\circ T (x_n) \in \D_{m+n+2}$.  Using \ref{I:L:DYNA ordering of x_n y_n, Tf > fT} above, $T\circ f(x_n) \leq f\circ T(x_n)$, so $f(x_n) \in T^{-1}(\D_{m+n+2})$ which completes the induction and the result now follows.
\end{proof}
The interaction between applications of $f$ and translations $T$ is illustrated in Figure \ref{F:DYNA mapping diagram}.  Lemma \ref{L:DYNA ordering of x_n y_n} implies $f^n(x_0) < f^n(y_0) < T\circ f^{n-1}(x_0)$.
\begin{figure}[H]
\centering
\begin{tikzpicture}
% real axis
\draw (0,0) -- (12.56,0) node[below]{$\Real$};
% circles radii reducing at 2,6,10,14
\draw (1.6,0) circle (1.2) node[above=1.12cm] {$\D_{m}$};
\draw (4.8,0) circle (1.12) node[above=1.12cm] {$\D_{m+1}$};
\draw (8,0) circle (1.04) node[above=1.12cm] {$\D_{m+2}$};
\draw (11.2,0) circle (0.96) node[above=1.12cm] {$\D_{m+3}$};
% x_n =2.08, 5.04, 9.8, 10.48
\draw (2.08,0.08)--(2.08,-0.08) node[below left=-0.8] {$x_0$};
\draw (5.04,0.08)--(5.04,-0.08) node[below left=-0.8] {$x_1$};
\draw (7.84,0.08)--(7.84,-0.08) node[below left=-0.8] {$x_2$};
\draw (10.48,0.08)--(10.48,-0.08) node[below left=-0.8] {$x_3$};
% y_n =2.32, 5.28, 8.08, 10.72
\draw (2.32,0.08)--(2.32,-0.08) node[below=0.4] {$y_0$};
\draw (5.28,0.08)--(5.28,-0.08) node[below=0.4] {$y_1$};
\draw (8.08,0.08)--(8.08,-0.08) node[below=0.4] {$y_2$};
\draw (10.72,0.08)--(10.72,-0.08) node[below=0.4] {$y_3$};
% f arrows for x_n
\draw[-{Stealth}] (2.08,0)  to [bend left=80] node[midway,above] {$f$} (5.04,0) ;
\draw[-{Stealth}] (5.04,0)  to [bend left=80] node[midway,above] {$f$} (7.84,0) ;
\draw[-{Stealth}] (7.84,0)  to [bend left=80] node[midway,above] {$f$} (10.48,0) ;
% f arrows for y_n
\draw[-{Stealth}] (2.32,0)  to [bend left=80] node[midway,above] {} (5.28,0) ;
\draw[-{Stealth}] (5.28,0)  to [bend left=80] node[midway,above] {} (8.08,0) ;
\draw[-{Stealth}] (8.08,0)  to [bend left=80] node[midway,above] {} (10.72,0) ;
% arrows T on top 
\draw[dashed, -{Stealth}] (2.08,0)  to [bend right=50] node[midway,below] {$T$} (5.28,0) ;
\draw[dashed, -{Stealth}] (5.04,0)  to [bend right=50] node[midway,below] {$T$} (8.24,0) ;
\draw[dashed, -{Stealth}] (7.84,0)  to [bend right=50] node[midway,below] {$T$} (11.04,0) ;
% Labels
\draw (0.4,-0.1) node[below=-2] {$2m\pi$};
\end{tikzpicture}
\caption{Illustration of mapping $f$ on the points $x_0, y_0$}\label{F:DYNA mapping diagram}
\end{figure}
One last preliminary is needed.  We can now estimate the hyperbolic distance between the points $x_n$ and $y_n$ with respect to the circular domain $\D_{m+n}$.
\begin{lemma}\label{L:DYNA hyperbolic distance between x_n and y_n}
For $m \geq N_0$ and $n \geq 1$, $T(x_{n-1}) \in \D_{m+n}$ and $\displaystyle \d [\D_{m+n}] {x_n, y_n} \leq \d [\D_{m+n}] {x_n, T(x_{n-1})}$.
\end{lemma}
\begin{proof}
	In this proof we use $\d{}$ and $\lambda$ to denote the hyperbolic distance and density respectively with respect to the disc $\D_{m+n}$.  Lemma \ref{L:DYNA ordering of x_n y_n} part \ref{I:L:DYNA ordering of x_n, y_n x_n in D_m+n and D_m+n+1 - 2pi} proved that $T(x_{n-1}) \in \D_{m+n}$. Moreover in Lemma \ref{L:DYNA ordering of x_n y_n} we showed that $x_n < y_n \leq T(x_{n-1})$.
	
	Within $\D_{m+n}$, real points lie on a diameter and therefore on a geodesic so that the hyperbolic distance is given by
		\begin{align*}
		    \d{x_n,y_n} = \int_{x_n}^{y_n} \lambda(z)\abs{dz} .
		\end{align*}
	For $n \geq  1$, because $\lambda$ is positive, $T(x_{n-1}) \in \D_{m+n} $ and $T(x_{n-1})> y_n$, we have
		\begin{align*}
		    \d {x_n,y_n}
		    &< \int_{x_n}^{T(x_{n-1})} \lambda(z)\abs{dz}
		    \\&=\d {x_n, T(x_{n-1})}. \qedhere
		\end{align*}
\end{proof}
These lemmas lead to the main result in this section, which was stated in the Introduction as part (5) of Theorem \ref{T:INTRO Main theorem}.
\begin{lemma}
The wandering domain $U_0$ is contracting.
\end{lemma}
\begin{proof}
	Choose fixed $m \geq N_0$.  
	
	We use the sequences $(x_n)$ and $(y_n)$ from Lemma \ref{L:DYNA ordering of x_n y_n}, which showed that $y_n > x_n$ and that $x_n, y_n \in \D_{m+n} \subset U_{m+n}$.  Since for any $k \geq 0$, $f^k(U_0) \subset U_k$ and a small real interval $(0, \delta)$ in $U_0$ maps to a small real interval in $U_k$, we can ensure $x_0, y_0$ have a pair of $k$-fold pre-images in $U_0$.  It is thus sufficient to show that the hyperbolic distance between iterates of two points in $U_m$ which do not eventually coincide has limit zero.  The classification of \cite[Theorem A]{benini+2021} then implies all such pairs of points have the same property.  
	
	Using the comparison principle \cite[Theorem 8.1]{beardon2006}, since for all $n \geq 0$ we have $\D_{m+n} \subset U_{m+n}$, the hyperbolic density for $U_{m+n}$ is less than that for $\D_{m+n}$, and so the same applies to hyperbolic distance.  Thus for $x_n,y_n \in \D_{m+n}$,
		\begin{align*}
		    \d [U_{m+n}] {x_n, y_n} \leq \d [\D_{m+n}] {x_n, y_n}
		\end{align*}
	and for $n \geq 1$, by Lemma \ref{L:DYNA hyperbolic distance between x_n and y_n} we then have
		\begin{align*}
		    \d [U_{m+n}] {x_n, y_n} \leq \d [\D_{m+n}] {x_n, Tx_{n-1}}. 
		\end{align*}
	Now $T^{m+n}$ is a M\"obius map from $D_{m+n}$ onto $\D_{m+n}$ and is therefore a hyperbolic isometry.  If for any $n \geq 0$ we write $T^{-(m+n)} (x_n) =\xi_n$ , then for $n \geq 1$
		\begin{align*}
			\d [\D_{m+n}] {x_n, Tx_{n-1}} 
			&= \d [D_{m+n}] {T^{-(m+n)}x_n, T^{-(m+n)}(T x_{n-1})} 
		    \\&= \d [D_{m+n}] {\xi_n, \xi_{n-1}} .
		\end{align*}
	Take $t_0$ of Lemma \ref{L:DYNA distance between iterates in H_m} ($n \geq 0$) to be $1 / \xi_0$, so that $1 / \xi_n = t_n$.  The M\"obius mapping $t = 1/z$ between $D_{m+n}$ and $H_{3(m+n)\pi}$ is also hyperbolic isometry so that part \ref{I:L:DYNA Distance between iterates hyperbolic distance t_n to t_n+1 in H_m} may be applied and as $n \to \infty$
		\begin{align*}
		    \d [D_{m+n}] {\xi_n, \xi_{n-1}} = \d [H_{3(m+n)\pi}] {t_n, t_{n-1}} \to 0.
		\end{align*}
	Therefore $d_{U_{m+n}}(x_n, y_n) \to 0$ as $n \to \infty$ and the wandering domain is contracting.
\end{proof}
\section{Conclusion}\label{S:CONCLUDE}
We have concentrated exclusively on a single function $f(z) = z\cos z + 2\pi$ and shown it has an orbit of parabolic like wandering domains $(U_n)$, bounded and shrinking which, scaled by $n$ and suitably positioned, converge to the cauliflower.  We have also shown that within the classification of \cite{benini+2021} the wandering domains are contracting and iterated points converge to the boundary.

It is natural to ask whether other similar functions exist.  We have begun to examine a family of functions
	\begin{align*}
	    f_\lambda (z) = z\cos z + \lambda \sin z + 2\pi, \quad \lambda \in \Complex.
	\end{align*}
The function already studied corresponds to $\lambda = 0$.  If the expansion \eqref{E:CONV power series for g_n} for $g_{n}$ that followed Definition \ref{D:CONV definitions of phi_m,n and g_m} is adapted to this family, it becomes
	\begin{align*}
	    g_{\lambda,n}(z) 
	    	&= (n+1) T^{-(n+1)}\circ f_\lambda \circ T^n \left( z / n\right) 
			\\&= z(1+\lambda) - \pi z^2 
			\\&\hspace{1cm}		+ \frac{1+\lambda}{n} z- \frac{\pi}{n} z^2 
								- \frac{n + 1 + \lambda}{2!n^3} z^3 
								+ \frac{2n(n+1)\pi}{4!n^4} z^4 
								+ \cdots.
	\end{align*}
As $n \to \infty$ all but the leading term $q_\lambda(z) = z(1+\lambda) - \pi z^2$ converges to zero, and we note that $q_\lambda$ is conjugate to the Mandelbrot set quadratic $z \mapsto z^2+c$ with $c = \frac{1}{2}(1+\lambda) - \frac{1}{4}(1+\lambda)^2$.  

It is plausible, as we found for $\lambda = 0$, that scaled iterates $nf_\lambda^n(z) $ behave like iterates of the quadratic $q^n_\lambda$.  Indeed much of the argument presented above is likely to carry through with little change for some values of $\lambda$, leading to a new orbit of non-congruent wandering domains converging to a suitable Fatou component of the quadratic.  Experimentation adds weight to this idea.  For example, taking $\lambda = \frac{1}{3} $ we obtain empirically a diagram almost identical to Figure \ref{F:JULIASET wandering domain} in which the wandering component suitably scaled appears to converge to the shape of the bounded Fatou component of the corresponding quadratic $\frac{4}{3}z-\pi z^2$.  

But further experimentation suggests the conclusion could remain valid for a wider range of $\lambda$.  Providing just one example, let $1+\lambda = e^{2\sqrt 2 \pi i}$ (Mandelbrot parameter $c \simeq -0.547+0.477i$).  In this case the anticipated limiting quadratic has multiple bounded Fatou components, among them a Siegel disc.  Using computer modelling it seems (within the empirical limitations) that the function $f_\lambda$ also has a wandering domain whose shape converges to that of the new quadratic.  The process is illustrated in Figure \ref{F:CONLUDE side by side}.  However, there are now multiple Fatou components, iterates no longer converge to zero and the techniques used above to prove convergence cannot be applied. 
\begin{figure}[H]
\centering
	\includegraphics[width=0.2\textwidth]{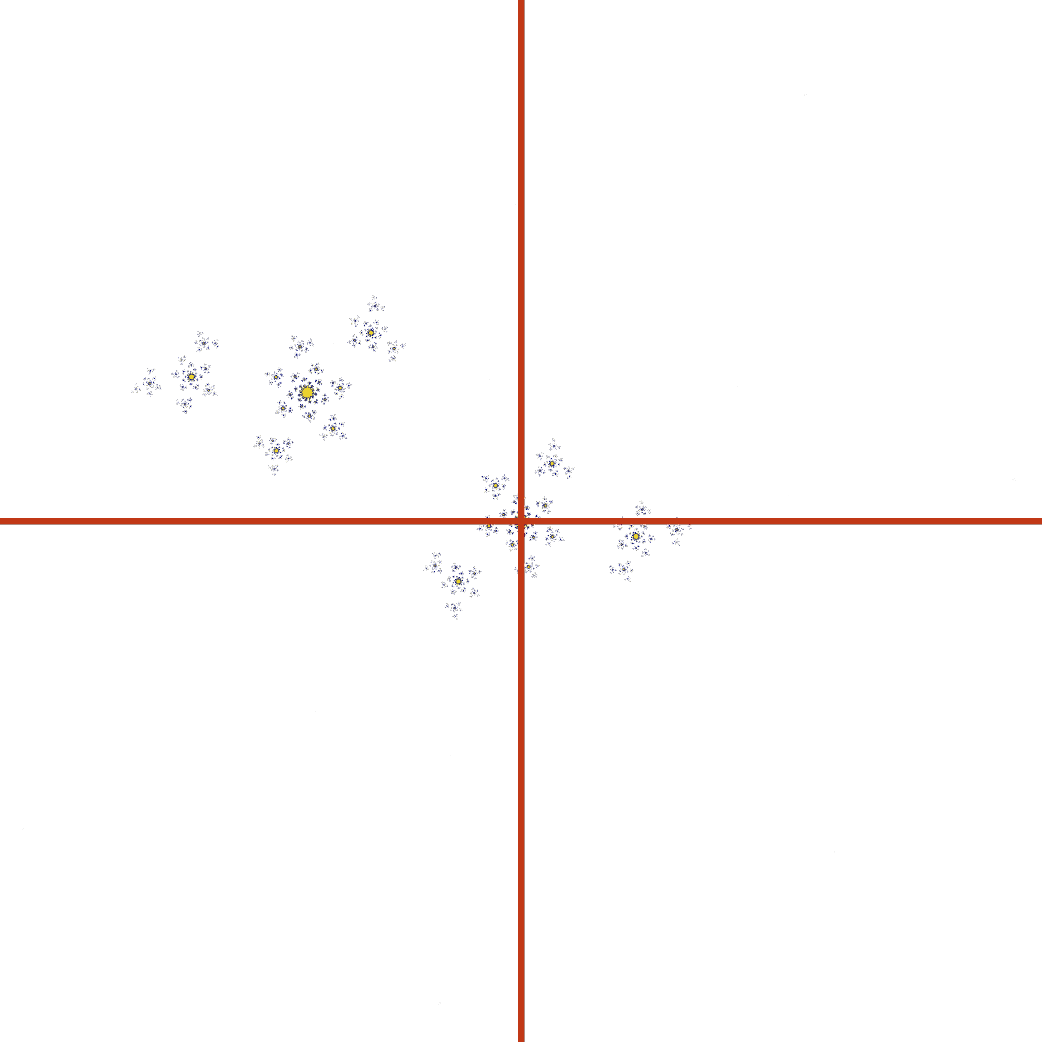}~ 
	\includegraphics[width=0.2\textwidth]{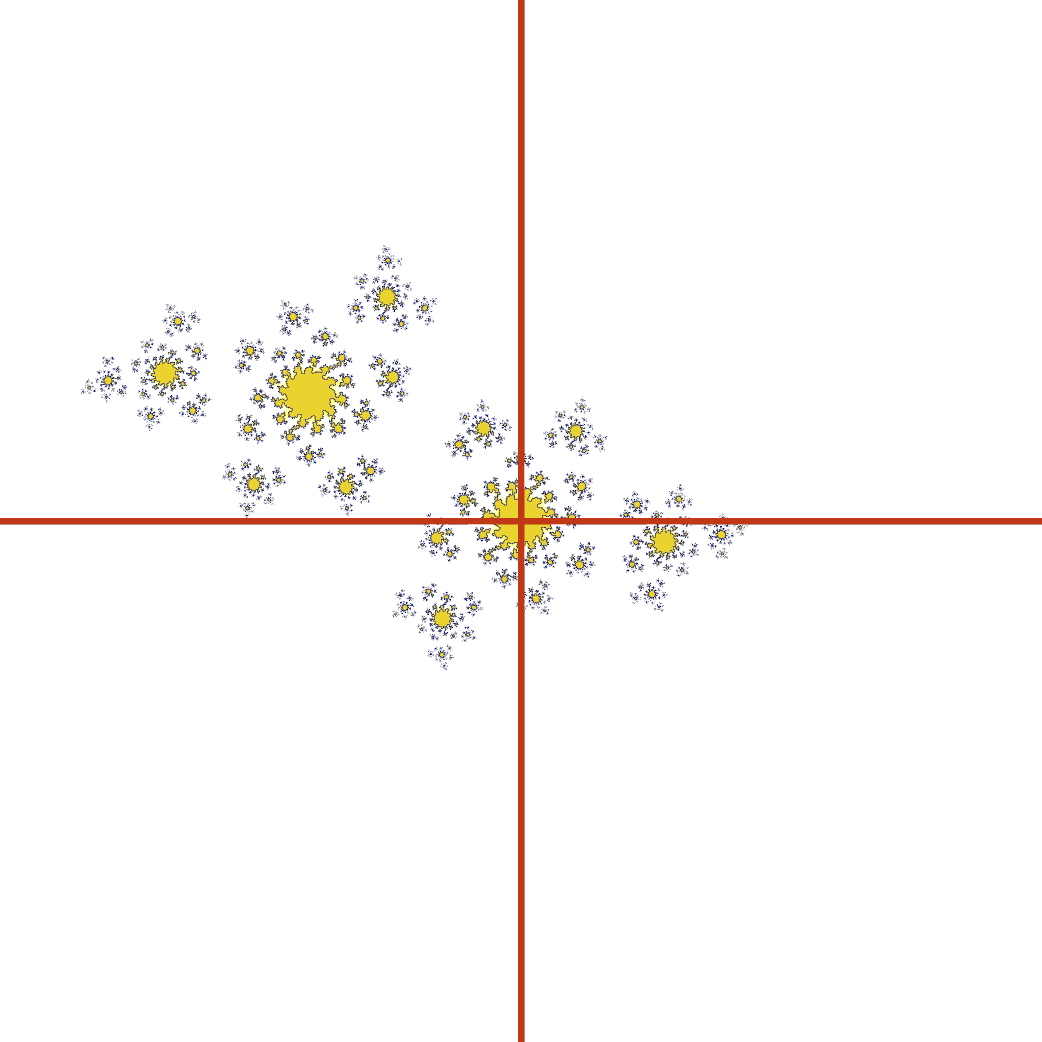}~
	\includegraphics[width=0.2\textwidth]{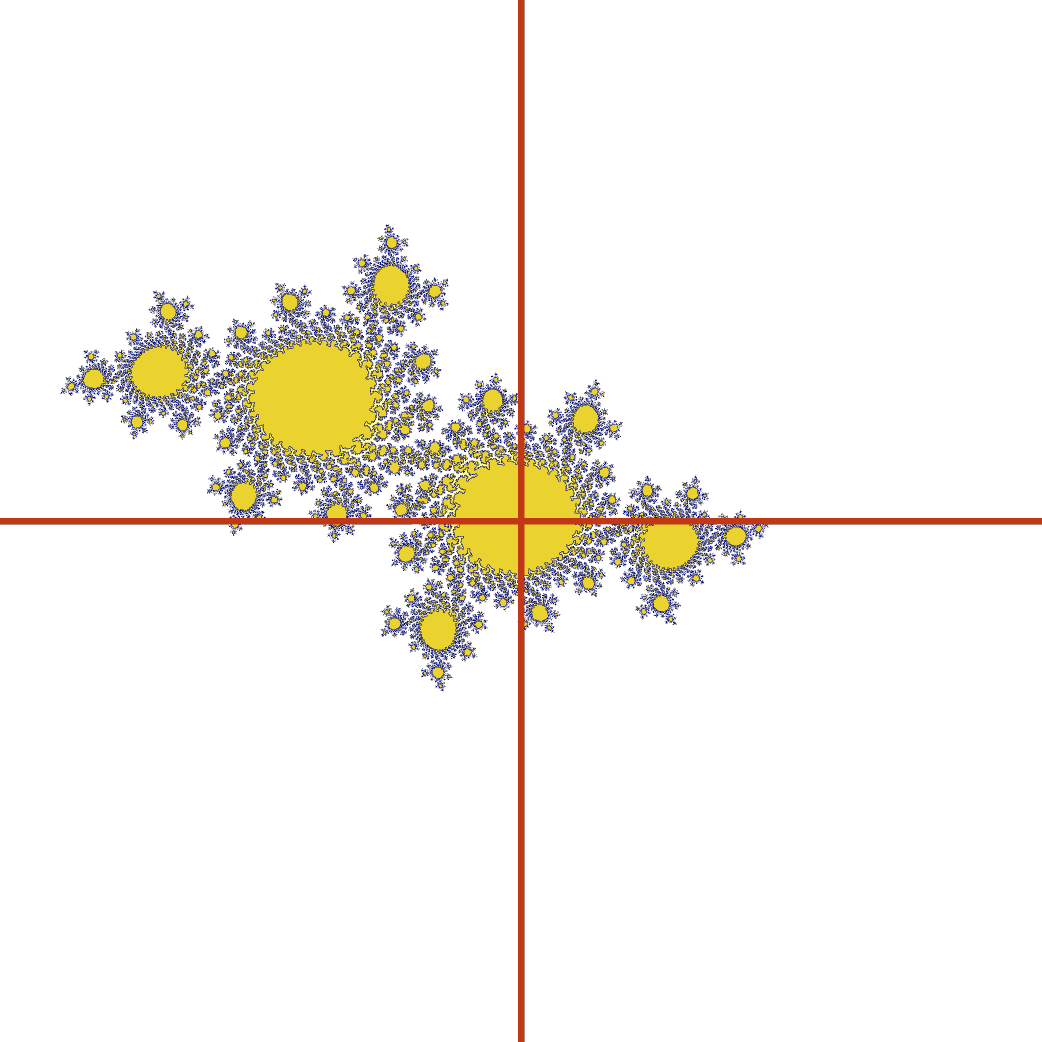}~
	\includegraphics[width=0.2\textwidth]{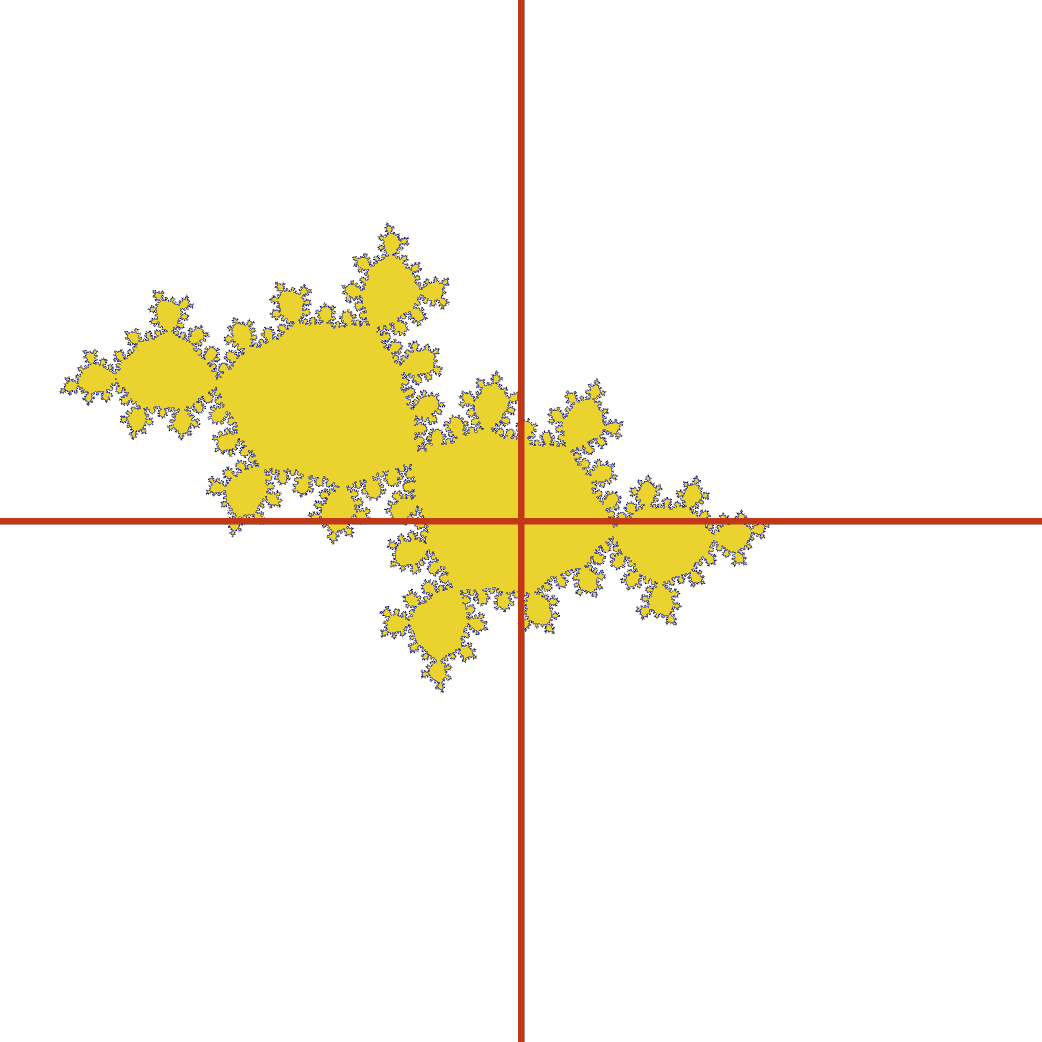}
\caption{Scaled and relocated approximate Fatou components (yellow with dark edging) of $f_\lambda$ for $\lambda= e^{2\sqrt 2 \pi i }-1$, left to right, 2nd, 10th, 200th and limiting iterates.}\label{F:CONLUDE side by side}
\end{figure}
We plan to study this family of function and hope to obtain similar conclusions to those when $\lambda=0$.  For some values of $\lambda$, this may be a modest extension of the work here, but we expect new techniques to be required if it is to be extended more generally.  These could then be applied not only to establish the asymptotic properties of wandering domains in this family, but also to analyse the behaviours in other examples, such as the non-congruent escaping wandering domains of transcendental self-maps of $\Complex \setminus \{0\}$ described in \cite{marti-pete2021}.

% \bibliographystyle{amsplain}

\printbibliography
\end{document}